\documentclass[11pt,letterpaper]{amsart}

\usepackage{amsmath,amscd,amssymb,amsthm,array,bbold,colortbl,enumitem,eucal}
\usepackage{epsfig,eucal,graphics,multicol,verbatim}
\usepackage{url}
\usepackage{color,xcolor}
\usepackage{eurosym}
\usepackage{appendix}
\usepackage{wrapfig}
\usepackage{multirow}
\usepackage{rotating}
\usepackage{slashbox}
\usepackage{tikz}
\usepackage{stmaryrd}
\usepackage{dsfont}
\usepackage{calrsfs}
\usepackage{mathabx}

\usetikzlibrary{matrix,arrows}

\theoremstyle{plain}
\newtheorem{theorem}{Theorem}[section]
\newtheorem{lemma}[theorem]{Lemma}

\newtheorem{corollary}[theorem]{Corollary}

\theoremstyle{definition}
\newtheorem{definition}[theorem]{Definition}
\newtheorem{notation}[theorem]{Notation}

\theoremstyle{remark}

\newtheorem{question}[theorem]{Question}

\numberwithin{equation}{section}

\begin{document}
\title[Lipschitz subactions for a flow]{Lipschitz sub-actions for locally maximal hyperbolic sets of a  $C^1$ flow
%\\ \medskip \small Sous actions Lipschitz pour des ensembles hyperboliques 
%localement maximaux de flots $C^1$
}

\author{Xifeng Su}
\address{School of Mathematical Sciences,
Beijing Normal University,
No. 19, XinJieKouWai Street, HaiDian District,
 Beijing 100875, P. R. China}
\email{xfsu@bnu.edu.cn,  billy3492@gmail.com}

\author{Philippe Thieullen}
\address{Institut de Math\'ematiques de Bordeaux,
Universit\'e de Bordeaux,
351 cours de la Lib\'eration, F 33405 Talence, France}
\email{philippe.thieullen@u-bordeaux.fr}

\keywords{Anosov  flow, weak KAM solutions, Lax-Oleinik operator, Lipschitz coboundary, positive Liv\v{s}ic criterion}.

\date{\today}

\newcommand{\Lip}{\text{\rm Lip}}
\newcommand{\Grass}{\text{\rm Grass}}
\newcommand{\Graph}{\text{\rm Graph}}
\newcommand{\diam}{\text{\rm diam}}
\newcommand{\dist}{\text{\rm dist}}
\newcommand{\Id}{\text{\rm Id}}
\newcommand{\Card}{\text{\rm Card}}
\newcommand{\Supp}{\text{\rm supp}}
\newcommand{\Osc}{\text{\,\rm Osc}}

\begin{abstract}
Liv\v{s}ic theorem for flows  asserts that a Lipschitz observable that has zero mean average along every periodic orbit is necessarily a coboundary, that is the Lie derivative of a Lipschitz function smooth along the flow direction. The positive Liv\v{s}ic theorem bounds from below the observable by such a coboundary as soon as the mean average along every periodic orbit  is non negative. Previous proofs give a H\"older coboundary. Assuming that the dynamics is given by a locally maximal hyperbolic flow, we show that the  coboundary can be Lipschitz. We introduce a new tool: the  Lax-Oleinik semigroup, inspired by Fathi's weak KAM theory.
%\bigskip\newline
%R\'esum\'e : Le th\'eor\`eme de Liv\v{s}ic pour les flots d'Anosov 
%affirment que toute observable Lipschitz qui poss\`ede une moyenne nulle 
%sur tout orbite p\'eriodique est n\'ecessairement un cobord, 
%c'est-\`a-dire la d\'eriv\'ee de Lie d'une fonction Lipschitz lisse le 
%long du flot. Le th\'eor\`eme de Liv\v{s}ic positif permet de minorer 
%toute obsersable par un cobord d\`es que sa moyenne est positive ou 
%nulle sur toute orbite p\'eriodique. Il est connu que la r\'egularit\'e 
%du cobord dans le cas du th\'eor\`eme de Liv\v{s}ic positif est au moins 
%H\"older. Nous montrons qu'en fait la r\'egularit\'e est Lipschitz. Nous 
%\'etendons ce r\'esultat aux flots pr\'eservant un ensemble hyperbolique 
%localement maximal. La preuve utilise des techniques  KAM faible et un 
%nouvel  op\'erateur de type Lax-Oleinik.
\end{abstract}

\maketitle
\tableofcontents

\section{Introduction and main results}

A $C^r$ flow, $r\geq1$, is a triple $(M,V,f)$ where $M$ is a $C^{r+1}$ manifold without boundary  of dimension $d+1\geq3$, not necessarily compact, equipped with a smooth Riemannian metric,  $V:M \to TM$ is a complete $C^{r}$ vector field that never vanishes, $f=(f^t)_{t\in\mathbb{R}}$ is a flow  with infinitesimal generator $V$, that is  a one-parameter group of $C^{r}$ bijective maps satisfying 
\[
\forall\, s,t \in\mathbb{R}, \ \forall\, x \in M, \  \frac{d}{dt}f^t(x) = V \circ f^t(x), \ f^0(x)= x, \quad f^{s+t} = f^s \circ f^t.
\]

\begin{definition}
\label{Definition:LocallyHyperbolicFlow}
Let $(M,V,f)$ be a $C^r$ flow, $r\geq1$, and $\Lambda \subseteq M$ be a  compact invariant set ($\forall\, t \in\mathbb{R}, \ f^t(\Lambda) = \Lambda)$. 
\begin{enumerate}
\item $\Lambda$ is said to be  {\it hyperbolic}  if there exist constants $C_\Lambda\geq1$, $\lambda^s < 0 < \lambda^u$, and a continuous splitting of $\Lambda$, $T_xM =E_\Lambda^u(x) \oplus E_{\Lambda}^0(x) \oplus E_\Lambda^s(x)$ for every $x \in\Lambda$, $d^u:=\dim(E^u_\Lambda) \geq1$, $d^s := \dim(E^s_\Lambda)$, such that, 
\begin{enumerate}
\item the splitting is equivariant: for every $x \in \Lambda$,
\begin{gather*}
T_xf^t(E^u(x)) = E^u(f^t(x)), \ T_xf^t(E^s(x)) = E^s(f^t(x)), \\
T_xf^t(V(x)) = V \circ f^t(x),
\end{gather*}
\item the splitting is transversally hyperbolic
\begin{gather*}
\renewcommand{\arraystretch}{1.2}
\forall x \in \Lambda, \ \forall t\geq0, \ 
\left\{\begin{array}{l}
\forall v \in E_\Lambda^s(x) ,\ \| T_xf^t(v) \| \leq C_{\Lambda} \, e^{t\lambda^s} \|v\|, \\
E_\Lambda^0(x) = V(x) \mathbb{R}, \quad  \\
\forall v \in E_\Lambda^u(x), \ \| T_xf^{t}(v) \| \geq C_{\Lambda}^{-1} \, e^{t\lambda^u} \|v\|.
\end{array}\right.
\end{gather*}
\end{enumerate}
\item $\Lambda$ is said to be {\it locally maximal}  if there exists an open neighborhood $U \supseteq \Lambda$ of compact closure such that $\Lambda = \bigcap_{t \in\mathbb{R}} f^t(U)$.
\end{enumerate}
\end{definition}

We extend in the following definition the notion of ergodic minimizing  value and the notion of  subaction. See  Garibaldi \cite{Garibaldi2017} or Jenkinson \cite{Jenkinson2019} for a complete review on the subject.  We recall that the Lie derivative of a Lipschitz function $u:M\to\mathbb{R}$ that is differentiable along the flow generated by a vector field $V$ is the function,
\[
\mathcal{L}_V[u](x) = \frac{d}{dt}\Big|_{t=0}u\circ f^t(x).
\]

\begin{definition}\label{Definition:ContinuousErgodicOptimization}
Let $(M,V,f)$ be a $C^1$  flow, $\Lambda \subseteq M$ be a compact invariant set, $U \supseteq \Lambda$ be an open neighborhood of $\Lambda$, and $\phi:U \to \mathbb{R}$ be a $C^0$ bounded function.
\begin{enumerate}
\item \label{Label:ContinuousErgodicOptimization_1} The {\it ergodic minimizing value} of $\phi$ (restricted to $\Lambda$) is the quantity
\[
\bar\phi_\Lambda := \lim_{t \to+\infty} \inf_{x \in\Lambda}\frac{1}{t}  \int_0^t \!\phi \circ  f^s(x) \, ds.
\]
\item \label{Label:ContinuousErgodicOptimization_2} A  continuous  function $u:U \to \mathbb{R}$ is said to be an {\it integrated  subaction of $\phi$ on $(U,\Lambda)$} if for every $x \in U$ and $t>0$ such that $\forall\, s \in [0,t], \ f^s(x) \in U$,
\begin{equation*} \label{Equation:LipschitzIntegratedSubaction}
u \circ f^t(x) -u(x) \leq \int_0^t (\phi \circ f^s(x) -\bar \phi_\Lambda) \, ds.
\end{equation*}
\item \label{Label:ContinuousErgodicOptimization_3} A  continuous  function $u:U \to \mathbb{R}$ is said to be a  {\it   subaction of $\phi$ on $(U,\Lambda)$} if $u$ is differentiable along the flow, $\mathcal{L}_V[u]$ is $C^0$, and
\[
\forall\, x \in U, \ \phi(x) -\bar \phi_\Lambda \geq  \mathcal{L}_V[u](x).
\]
If $u$ and $\mathcal{L}_V[u]$are both Lipschitz continuous, we say that $u$ is a {\it Lipschitz continuous subaction}.
\end{enumerate}
\end{definition}

Notice that, if we choose a fixed time of iteration $\tau>0$, if we define a discrete map by $F:=f^\tau$ and an integrated observable by $\Phi(x) := \int_0^{\tau} \phi \circ f^t(x) \, dt$, if  
\[
\bar\Phi_\Lambda := \tau \bar\phi_\Lambda = \lim_{n\to+\infty} \frac{1}{n} \inf_{x \in\Lambda} \sum_{k=0}^{n-1} \Phi \circ F^k(x),
\] 
an integrated subaction $u$ for $\phi$ with respect to the flow is also a subaction for $\Phi$ (item ii of Definition 1.2 of \cite{SuThieullenYu2021}) with respect to the map $F$
\begin{gather*}
\forall\, x \in\Lambda, \ \Phi(x) -\bar\Phi_\Lambda \geq u \circ F(x) -u(x).
\end{gather*}

Our main result is the following.

\begin{theorem} \label{Theorem:ContinuousSubactionExistence}
Let $(M,V,f)$ be a $C^1$ flow, $\Lambda \subseteq M$ be a locally maximal hyperbolic compact connected invariant set. Then there exist an open neighborhood $\Omega$ of $\Lambda$, a constant $K_\Lambda>0$, such that, for every  Lipschitz continuous function $\phi:M \to \mathbb{R}$, there exists  $u : \Omega \to \mathbb{R}$ on $(\Omega,\Lambda)$  satisfying
\begin{enumerate}
\item \label{Item:ContinuousSubactionExistence_1} $u$ is Lipschitz continuous, $\Lip(u) \leq K_\Lambda \Lip(\phi)$,
\item \label{Item:ContinuousSubactionExistence_2} $u$ is differentiable along the flow,
\item \label{Item:ContinuousSubactionExistence_3} $\mathcal{L}_V[u]$ is Lipschitz continuous, $\Lip(\mathcal{L}_V[u]) \leq K_\Lambda \Lip(\phi)$,
\item \label{Item:ContinuousSubactionExistence_4} $\forall\, x \in \Omega, \ \phi(x) - \bar \phi_\Lambda \geq \mathcal{L}_V[u](x)$.
\end{enumerate}
\end{theorem}

\begin{proof}
The proof readily follows from  the Theorems \ref{Theorem:ContinuousPositiveCriteria} and \ref{Theorem:ContinuousWeakKAMsolution}, and a smoothing technique using a regularizing operator. The proof is done at the end of section~\ref{Section:LaxOleinikSemigroup}.
\end{proof}

The existence of a Lipschitz subaction is proved under the hypothesis that the observable $\phi$ is Lipschitz and the flow $f$ is $C^1$. We could expect a more regular subaction under the hypothesis of a $C^2$ flow and a $C^2$ observable. Actually it is false in general as it was observed by Bousch-Jenkinson \cite{BouschJenkinson2002} for an example of a trigonometric polynomial under the doubling map on the circle. The Lipschitz regularity is in some sense optimal for a general statement

Notice that we recover the positive Liv\v{s}ic theorem mentioned in the abstract. For  hyperbolic compact sets, the set $\mathcal{P}(\Lambda,f)$ of periodic measures $\frac{1}{T} \int_0^T \delta_{f^s(x)}\,ds$, where $x \in\Lambda$ and $f^T(x)=x$ is a periodic point of least period $T$, is dense in the space of invariant probability measures $\mathcal{M}_1(\Lambda,f)$  supported in $\Lambda$. The ergodic minimizing value admits then the equivalent definition
\begin{gather*}
\bar \phi_\Lambda = \inf_{\mu \in \mathcal{M}_1(\Lambda,f)} \int \phi \, d\mu = \inf_{\mu \mathcal{P}(\Lambda,f)} \int \phi \, d\mu.
\end{gather*}

The ergodic maximizing value of $\phi$ is defined similarly by
\begin{gather*}
\bar{\bar\phi}_\Lambda := \lim_{t \to+\infty} \sup_{x \in\Lambda}\frac{1}{t}  \int_0^t \!\phi \circ  f^s(x) \, ds \geq \bar\phi_\Lambda.
\end{gather*}
A bound from above $\phi - \mathcal{L}_V[u] \leq  \bar{\bar\phi}_\Lambda$ is obtained similarly. Actually it is easy to obtain both bounds simultaneously.

\begin{corollary} \label{Corollary:SimultaneousBounds}
Let $(M,V,f,\Lambda,\phi)$  as in Theorem \ref{Theorem:ContinuousSubactionExistence}. Assume $\bar\phi_\Lambda < \bar{\bar\phi}_\Lambda$.
Then there exists a Lipschitz continuous subaction $u : \Omega \to \mathbb{R}$ (satisfying \ref{Item:ContinuousSubactionExistence_1}--\ref{Item:ContinuousSubactionExistence_3} of Theorem \ref{Theorem:ContinuousSubactionExistence})  defined in a neighborhood of $\Lambda$ such that
\begin{gather*}
\bar\phi_\Lambda \leq \phi - \mathcal{L}_V[u] \leq  \bar{\bar\phi}_\Lambda.
\end{gather*}
\end{corollary}

In the case $M$ is compact and $f$ is a transitive Anosov flow, if $\bar\phi_M = \bar{\bar\phi}_M$, the classical Liv\v{s}ic Theorem asserts that $\phi$ is a regular coboundary in the sense $\phi - \mathcal{L}[u] = \bar\phi_\Lambda$ where $u:M\to\mathbb{R}$ possesses the same regularity as $f$ (see \cite{LlaveMarcoMoriyon1986}). Corollary \ref{Corollary:SimultaneousBounds} is thus also valid in the case $\bar\phi_M = \bar{\bar\phi}_M$. Nevertheless we do not assume in the present paper that $f$ is transitive on $\Lambda$ and we do not know whether  Corollary \ref{Corollary:SimultaneousBounds} is still true in the case $\bar\phi_M = \bar{\bar\phi}_M$.

The plan of the article is as follows. In section \ref{Section:PositiveLivsicCriteria},  we show that for locally maximal hyperbolic compact invariant set $\Lambda$, every Lipschitz continuous function $\phi$ satisfies a criterion  called   ``positive Liv\v{s}ic criterion''.  In section \ref{Section:LaxOleinikSemigroup}, we prove Theorem \ref{Theorem:ContinuousSubactionExistence} without assuming  that $\Lambda$ is hyperbolic: we show that any Lipschitz continuous function satisfying the positive Liv\v{s}ic criterion admits a Lipschitz continuous subaction. We introduce a nonlinear semigroup analogous to the  ``Lax-Oleinik semigroup'' in \cite{Fathi2016}. A fixed point of the Lax-Oleinik semigroup provides a Lipschitz continuous integrated subaction. We conclude the proof by using a regularizing operator.  

The interest of splitting the proof in two distinct parts leads to the existence of a Lipschitz continuous subaction even in the case the flow is not uniformly hyperbolic, nor does it satisfy the shadowing lemma. The Lax-Oleinik semigroup gives an explicit construction of the subaction in addition to more regularity. We obtain a Lipschitz regularity contrary to the previous proofs that give only H\"older regularity (see  Lopes, Roasa, Ruggiero \cite{LopesRoasaRuggiero2007}, Lopes, Thieullen \cite{LopesThieullen2005}, or Pollicott, Sharp \cite{PollicottSharp2004}). In both papers, the stable manifolds are used to construct the subaction and it is known that the stable manifolds is  in general H\"older and not Lipschitz. Huang-Lian-Ma-Xu-Zhang \cite{HuangLianMaXuZhang2019_2}   obtained a Lipschitz continuous integrated  subaction $u_T : \Omega \to \mathbb{R}$, for $T$ sufficiently large in the sense
\begin{equation*}
\forall\, x \in \Omega,  \ \int_0^T \phi \circ f^s(x) \, ds \geq u_T \circ f^T(x) -u_T(x) + T\bar \phi_\Lambda.
\end{equation*}
There is no reason in their paper that $u_T$ may be chosen independently of $T$. We show that $u$ may be chosen independently of $T$ and that the above inequality may be differentiated in $T$. 

\begin{question}
\begin{enumerate}
\item The Mather beta  function is  obtained by minimizing an action of an arbitrarily long path with a fixed homology. The Mather alpha function is also obtained by minimizing the action of an  observable  modified by a cohomology (to be clarified in our setting). As our main tool is a notion of penalized action  of an observable (see \eqref{Equation:WeightedAction}), it would be interesting to extend both Mather beta and alpha functions to that framework and see how they are related to the stable norm, (see \cite{PollicottSharp2004}, \cite{Massart2011}).
\item An approximate positive Liv\v{s}ic of error $\epsilon$ could be expected as in \cite{Katok1990} and \cite{GouezelLefeuvre2021}. The ergodic minimizing value $\bar \phi_\Lambda$ is replaced by the infimum of $\frac{1}{T}\int_0^T \phi \circ f^s(x) \, ds$ over all periodic point $x$ of period least $T \leq \epsilon^{-1}$. In these papers an approximate Liv\v{s}ic theorem is proved and the coboundary is H\"older continuous and depends on $\epsilon$. It would be interesting to extend their result for positive Liv\v{s}ic theorem with a better regularity.
\end{enumerate}
\end{question}

%%%%%%%%%%%%%%%%%%%%%%%
\section{The positive Liv\v{s}ic criterion}
\label{Section:PositiveLivsicCriteria}
%%%%%%%%%%%%%%%%%%%%%%%

The {\it  positive Liv\v{s}ic criterion} below is a new concept similar to the notion of the (discrete) positive Liv\v{s}ic criterion introduced in Definition 3.2 of  \cite{SuThieullenYu2021} or the  notion of rectifiable orbit introduced in  Lemma 1.1 in \cite{Bousch2011}. The fact that a Lipschitz  function on a locally maximal compact invariant hyperbolic set satisfies the positive Liv\v{s}ic criterion is the key ingredient of the proof of the existence of a Lipschitz subaction (instead of a H\"older subaction as in \cite{LopesThieullen2005}). The proof of this criterion is  strongly related to a version of the shadowing lemma in the continuous setting,  to the fact that a pseudo orbit (defined in the continuous setting in Definition~\ref{Definition:PeriodicPseudoOrbitFlow}) is shadowed by a true orbit.
 
The {\it penalized action of a  piecewise $C^1$ continuous  path $z:[0,T] \to M$},  with  penalized constant $C\geq0$,  is the real number given by
\begin{gather}
\mathcal{A}_{\phi,C}(z)  := \int_0^T \big[(\phi  - \bar\phi_\Lambda )\circ z(s)+ C \|V \circ z(s) -z^{\prime} (s)\| \big] \, ds. \label{Equation:WeightedAction}
\end{gather}

Notice that the second term in the penalized action, $L(x,v) := \|V(x)- v\|$, is similar to the standard Lagrangian term, $L(x,v) = \frac{1}{2} \| V(x) - v\|^2$, used  to embed the flow of a vector field $V$ (see (1.20) in  \cite{ContrerasIturriaga1999}) in a Lagrangian problem.

\begin{definition}[Positive Liv\v{s}ic criterion]
\label{Definition:PositiveLivcicCriteria}
Let $(M,V,f)$ be a $C^1$ flow, $\Lambda$ be a compact connected invariant set,  $\Omega \supset \Lambda$ be an open neighborhood of $\Lambda$ with compact closure, and $C\geq0$ be a non negative constant.  Let  $\phi: \Omega \to \mathbb{R}$ be a $C^0$ bounded function  and $\bar\phi_\Lambda$ be the ergodic minimizing value of $\phi$ restricted to $\Lambda$. We say that $\phi$ satisfies the {\it positive Liv\v{s}ic criterion on $(\Omega,\Lambda)$ with penalized  constant $C$} if 
\[
\inf_{T>0} \, \inf_{z:[0,T] \to \Omega}\mathcal{A}_{\phi,C}(z) > -\infty,
\]
where the infimum is realized over $T>0$ and the set of  piecewise $C^1$ continuous paths $z:[0,T] \to \Omega$.
\end{definition}

The following lemma justifies the introduction of the positive Liv\v{s}ic criterion in the case  there exists a $C^1$ subaction.

\begin{lemma} \label{Lemma:Justification}
Let $(M,V,f)$ be a $C^1$  flow and $\phi:M \to \mathbb{R}$ be a $C^0$ bounded function. Assume there exists a $C^1$ subaction  $u:M \to \mathbb{R}$,  $\phi - \bar\phi_\Lambda \geq \mathcal{L}_{V} [u]$. Then $\phi$ satisfies the positive Liv\v{s}ic criterion with  $C := \|du\|_\infty$: for every $T>0$, for every  piecewise $C^1$  continuous   path $z:[0,T] \to M$
\[
\mathcal{A}_{\phi,C}(z) \geq -2 \Osc(u),
\]
where $\Osc(u) = \sup_{x\in M} u(x) - \inf_{x\in M} u(x)$.
\end{lemma}

\begin{proof}
For any $C^1$ function $u$, the Lie derivative admits the equivalent form
\[
\mathcal{L}_{V} [u](x) = du(x) \cdot V(x).
\]
Then 
\begin{align*}
\int_0^T \! \Big(\phi &\circ z(s) -\bar\phi_\Lambda +\| du\|_{\infty} \|V \circ z(s) -z^{\prime} (s)\| \Big)ds \\
\geq & \int_0^T \! \Big( du \circ z \cdot V \circ z +\| d u\|_{\infty} \|V \circ z -z^{\prime} \| \Big) \, ds \\
= &  \int_0^T \! \Big( du \circ z \cdot(V \circ z -z ^{\prime}) +\| d u\|_{\infty} \|V \circ z -z^{\prime} \| \Big)ds +  \int_0^T\!  du \circ z \cdot z^{\prime}  \, ds\\
\geq & \int_0^T du  \circ z \cdot z^{\prime} \, ds =  u \circ z(T) -u \circ z(0) \geq -2\|u\|_{\infty}.\qedhere
\end{align*}
\end{proof}

The following theorem shows that the conclusions of Lemma \ref{Lemma:Justification} are still valid without the assumption of the existence of a $C^1$ subaction.

\begin{theorem} \label{Theorem:ContinuousPositiveCriteria}
Let $(M,V,f)$ be a $C^1$ flow, $\Lambda$ be a locally maximal hyperbolic compact connected invariant set as in Definition \ref{Definition:LocallyHyperbolicFlow}. Then there exist an open neighborhood $\Omega$ of $\Lambda$ of compact closure and constants $C_\Lambda\geq0$, $\delta_\Lambda \geq0$, such that for every Lipschitz continuous $\phi : \Omega \to \mathbb{R}$, for every piecewise $C^1$ continuous path $z:[0,T] \to \Omega$
\[
 \int_0^T \big[(\phi  - \bar\phi_\Lambda )\circ z(s)+ C_\Lambda\Lip(\phi) \|V \circ z(s) -z^{\prime} (s)\| \big] \, ds \geq -\delta_\Lambda \Lip(\phi).
\]
\end{theorem}
The constant $C_\Lambda$ depends only on the hyperbolicity of $\Lambda$. The heart of the proof of Theorem \ref{Theorem:ContinuousPositiveCriteria} consists in writing the action $\mathcal{A}_{\phi,C}(z)$ as a Birkhoff sum of penalized actions under the dynamics of a sequence of Poincar\'e return maps and apply a version of the Anosov shadowing Lemma in local charts  (see   \cite[Theorem 2.1]{SuThieullenYu2021}).

We choose once for all  a family of adapted local flow boxes   
\begin{gather*}
\Gamma=(\Gamma,\Sigma, E, N,F,A)
\end{gather*}
as explained in Definition~\ref{Definition:AdaptedLocalFlowBoxes}. We recall that $\Gamma := (\gamma_x)_{x\in\Lambda}$  is a family of local diffeomorphisms  $\gamma_x : (-\tau,2\tau) \times B_x(\rho) \to M$  on a transverse ball $B_x(\rho) \subset \mathbb{R}^d$ of radius $\rho$ with respect to the norm $\| \cdot \|_x$; $\Sigma := (\Sigma_x)_{x \in\Lambda}$ is a family of local Poincar\'e sections passing through $\gamma_x(0,0)$ with return time $\tilde \tau_{x,y} : B_x(\rho) \to (0,2\tau)$ between $\Sigma_x$ and $\Sigma_y$ when $x$ and $y$ are forward admissible; $E := (E_x^{u/s})_{x \in\Lambda}$ is a family of splittings $\mathbb{R}^d = E^u_x \oplus E^s_x$ into unstable and stable directions equivariant by the local return maps 
\begin{gather*}
f_{x,y}(v) = \gamma_{y,0}^{-1} \circ f^{{\tilde\tau_{x,y}}(v)} \circ \gamma_{x,0}(v), \ \ \forall\, v \in B_x(\rho) \to B(1)
\end{gather*} 
and uniformly hyperbolic with respect to a family $N :=(\| \cdot \|_x)_{x \in\Lambda}$  of $C^0$ norms on $\mathbb{R}^d$ adapted to the splitting. We will need the notion of a  trap mechanism  explained in figure \ref{figure:LocalFlowBox} and Definition \ref{Definition:NotationsReducedPositiveLivsic}.

\begin{figure}[hbt]
\centering
\includegraphics[width=1\textwidth]{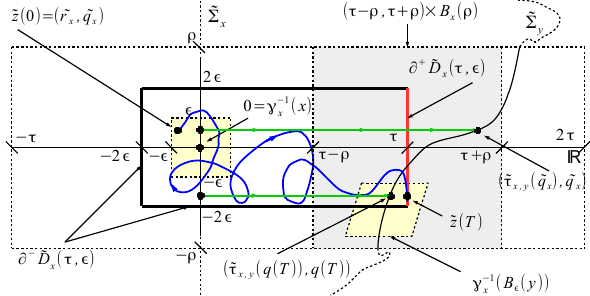}
\caption{The trap mechanism of size $(\tau,\epsilon)$. 
The drawing corresponds to the pseudo orbit case where the path in blue $s\mapsto \tilde z(s)$ exits at the forward boundary $\partial^+  \tilde D_x(\tau,\epsilon)$ in red. The two yellow regions correspond to the $\epsilon$-neighborhoods $U_x(\epsilon)$ and $ U_y(\epsilon)$ in $M$ containing respectively  $z(0)$ and $z(T)$. $\tilde\Sigma_x = \{0\} \times B(1)$ and $\tilde\Sigma_y = \gamma_x^{-1}(\Sigma_y)$ are the local Poincar\'e sections observed in the flow box $(-\tau,2\tau) \times B_x(\rho) $.}
\label{figure:LocalFlowBox}
\end{figure}

\begin{notation}[The trap mechanism] \label{Definition:NotationsReducedPositiveLivsic}
Let  $\epsilon < \frac{1}{2} \tau$.  We introduce the notion of  a {\it local trap box}  $D_x(\tau,\epsilon)$ of size $(\tau,\epsilon)$ located at some $x \in\Lambda$ is the open set
\begin{gather*}
D_x(\tau,\epsilon) = \gamma_x(\tilde D_x(\tau, \epsilon)) \ \text{where} \ \ \tilde D_x( \tau,\epsilon) :=(-2\epsilon,\tau) \times  B_x(2\epsilon)
\end{gather*}
equipped with the norm  $\|(r,u)\|_x = \max(|r|,\|u\|_x)$.  We introduce a {\it local ball} $U_x(\epsilon)$  of size $\epsilon$ through the chart $\gamma_x$ 
\begin{gather*}
U_x(\epsilon) := \gamma_x((-\epsilon,\epsilon) \times B_x(\epsilon)).
\end{gather*}
We choose $\epsilon$ small enough so that, if   $\gamma_x((\tau-\epsilon,\tau+\epsilon) \times B_x(2\epsilon)) $ intersects $U_y(\epsilon)$, then $(x,y)$ is $\Gamma$ forward admissible as in item \ref{Item:AdaptedLocalFlowBoxes_7bis} in  Notation \ref{Definition:AdaptedLocalFlowBoxes}, that is
\begin{gather*}
U_y(\epsilon) \subset \gamma_x((\tau-\rho,\tau+\rho) \times B_x(\rho)) \ \ \text{and} \ \ 
f_{x,y}(0) \in B_y(\epsilon(\rho)).
\end{gather*}
See item \ref{Item:AdaptedLocalFlowBoxes_1} of \ref{Definition:AdaptedLocalFlowBoxes} for the definition of $\epsilon(\rho)$. In particular, the local return time from $\Sigma_x \cap U_x(\epsilon)$ to $\Sigma_y$ belongs to the interval $(\tau-3\epsilon,\tau+3\epsilon)$. Let $\partial^+ \tilde D_x(\tau,\epsilon) \sqcup \partial^- \tilde D_x(\tau,\epsilon)$ be the {\it forward and backward  boundary}  of $\tilde D(\tau,\epsilon)$,
\begin{gather*}
\partial^+ \tilde D_x(\tau,\epsilon) := \{\tau\} \times B_x(2\epsilon), \\
\partial^- \tilde D_x(\tau,\epsilon) := \{-2\epsilon\} \times \overline{B_x(2\epsilon)} \cup [-2\epsilon, \tau] \times \partial B_x(2\epsilon), \\
\partial^\pm D_x(\tau,\epsilon) := \gamma_x(\partial^\pm \tilde D_x(\tau,\epsilon)).
\end{gather*}
We choose   a finite subset $\Lambda_*  \subset \Lambda$ such that $\Lambda \subset \bigcup_{x \in \Lambda_*}U_x(\epsilon)$ and $\epsilon_{AS} < \epsilon$  small enough so that 
\[
\Omega_{AS}:=\{x \in M : d(x,\Lambda) < \epsilon_{AS}\} \subset \bigcup_{x \in \Lambda_*}U_x(\epsilon) \subset U,
\]
where $U$  is the open set that defines the locally maximal set $\Lambda$ as in Definition \ref{Definition:LocallyHyperbolicFlow}. Let 
\[
\Lip(\Gamma) := \sup_{x \in \Lambda} \big(\Lip_x(\gamma_x),\Lip_x(\gamma_x^{-1}) \big).
\]
\end{notation}

We now consider a  piecewise $C^1$ continuous  path  $z: [0,T] \to M$ whose image lies in some $\Omega_{AS} \cap D_x(\tau,\epsilon)$, $x\in \Lambda_*$. We assume that the path $z$ starts close to $x$, 
\[
z(0) \in U_x(\epsilon).
\] 
We denote by $\tilde z := \gamma_x^{-1} \circ z$ the pull backward of $z$ by $\gamma_x$ and 
 discuss three cases. 
 
\begin{enumerate}
\item {\it The pseudo orbit case}: the path $z$ exits at the forward boundary  
\[
z(T) \in \partial^+ D_x(\tau,\epsilon).
\] 
\item {\it The escaping case}: the path exits at  the backward boundary 
\[
z(T) \in \partial ^- D_x(\tau,\epsilon).
\] 
\item {\it The trap case}: the path stays inside $ D_x(\tau,\epsilon)$. 
\end{enumerate}
In the first case, we  assume that there exists $y\in\Lambda_*$ such that
\[
z(T) \in U_y(\epsilon).
\] 
In that case $(x,y)$ is  $\Gamma$ forward admissible, the local return time is well defined 
\[
{\tilde\tau_{x,y}} : B_x(\rho) \to (0, 2\tau), \quad f^{\tilde \tau_{x,y}(q)} \circ \gamma_x(0,q) \in \Sigma_y.
\] 
The box coordinates (defined in different charts) of the two endpoints $z(0)$ and $z(T)$ are well defined
\begin{gather*}
z(0) = \gamma_x(\tilde r_x,\tilde q_x), \quad z(T) = \gamma_y(\tilde r_y,\tilde q_y).
\end{gather*}
The local Poincar\'e return map is also well defined
\[
f_{x,y}(q) = \gamma_{y,0}^{-1} \circ f^{{\tilde\tau_{x,y}}(q)} \circ \gamma_{x,0}(q) :   B_x(\rho) \to B(1),
\]
where $\gamma_{x,0}$ is the restriction of $\gamma_x$ to $\{0\} \times B_x(\rho)$.

 In the second and third cases we don't need  to assume the existence of a point $y\in\Lambda_*$ and to define a local return time.  See Figure \ref{figure:LocalFlowBox} for a schematic description of the trap mechanism.

The following lemma is crucial and gives a lower bound of $\mathcal{A}_{\phi,C}(z)$ in the three cases. The constant $\mathcal{A}$ is set arbitrarily in  the second item of the lemma but will be defined later in Lemma \ref{lemma:FinitenessPseudoBirkhoffIntegral}, $\mathcal{A} = \mathcal{A}_3^*$. A large value of $\mathcal{A}$ forces $C_2(\mathcal{A})$ to be large. 

\begin{lemma} \label{Lemma:PositiveLivsicChecking}
Let $x\in\Lambda$, $D_x(\tau,\epsilon)$, $U_x(\epsilon)$ defined in Notation~\ref{Definition:NotationsReducedPositiveLivsic}. Let $T>0$, and $z:[0,T] \to \Omega_{AS}\cap D_x(\tau,\epsilon)$ be a  piecewise $C^1$ continuous path such that $z(0) \in U_x(\epsilon)$. Let 
\begin{gather*}
C_1 := 6 \Lip(\phi)\Lip(\Gamma)(\tau \Lip(\Gamma)+\diam(\Omega_{AS})), \\
\mathcal{A}_1^* := 8\tau \Lip(\phi)(\tau\Lip(\Gamma)+\diam(\Omega_{AS})).
\end{gather*}

We discuss 3 cases:

\begin{enumerate}
\item \label{Item:PositiveLivsicChecking_1}The pseudo orbit case: $  z(T) \in \partial^+  D_x(\tau,\epsilon)$. The penalized action $\mathcal{A}_{\phi,C}(z)$ is bounded from below by 2  terms: the first term is a coboundary $\Psi_y - \Psi_x$, the second term is a penalized discrete action $\mathbb{A}_{\tilde C}(x,y)$ between the two Poincar\'e sections,. More precisely, let  $(\tilde r_x,\tilde q_x) = \gamma_x^{-1}(z(0))$, $(\tilde r_y,\tilde q_y) = \gamma_y^{-1}(z(T))$.  Then for every $C \geq C_1$,
\begin{gather*}
\mathcal{A}_{\phi,C}(z) \geq \Psi_y-\Psi_x+  \mathbb{A}_{\tilde C}(x,y)
\intertext{where $\tilde C :=  {C}/(\sqrt{8} \, \Lip(\Gamma)^3)$ and}
\mathbb{A}_{\tilde C}(x,y) := \Phi_{x,y}+  \tilde C \| f_{x,y}(\tilde q_x) - \tilde q_y\|_y, \\
\Phi_{x,y} := \int_{0}^{{\tilde\tau_{x,y}}(\tilde q_x)} (\phi-\bar\phi_\Lambda)   \circ \gamma_x(s,\tilde q_x) \, ds,\\
\Psi_x := \int_{0}^{\tilde r_x} (\phi-\bar\phi_\Lambda)   \circ \gamma_x(s,\tilde q_x) \, ds, \
\Psi_y := \int_{0}^{\tilde r_y} (\phi-\bar\phi_\Lambda) \circ \gamma_y(s,\tilde q_y) \, ds.
\end{gather*}

\item \label{Item:PositiveLivsicChecking_2} The escaping case: $z(T) \in \partial^-  D_x(\tau,\epsilon)$. The penalized action is bounded from below by a uniform positive value.   Let $\mathcal{A} >0 $ and 
\begin{gather*}
C_2(\mathcal{A}) := \max \Big( \frac{16}{3} \frac{\tau}{\epsilon}C_1, \frac{4}{\epsilon}\Lip(\Gamma) \mathcal{A} \Big).
\end{gather*} 
Then for every $C \geq C_2(\mathcal{A})$,
\begin{gather*}
\mathcal{A}_{\phi,C}(z) \geq \mathcal{A}.
\end{gather*}
\item \label{Item:PositiveLivsicChecking_3}  The trap case: $\forall\, t\in [0,T], \ z(t) \in  D_x(\tau,\epsilon)$. The penalized action may be non positive but is uniformly bounded from below. More precisely
\[
\mathcal{A}_{\phi,C}(z)  \geq - \mathcal{A}_1^*.
\]
\end{enumerate}
\end{lemma}

\begin{proof}
Let $\tilde z(s) =\gamma_x^{-1} \circ z(s)=(r(s),q(s))$ be the local coordinates of $\tilde z(s)$ for every $s\in[0,T]$. By definition $\tilde r_x = r(0)$ and $\tilde q_x = q(0)$.  Define the reduced observable
\begin{gather*}
\tilde\psi := (\phi-\bar\phi_\Lambda) \circ \gamma_x :  \tilde D_x(\tau,\epsilon) \to \mathbb{R}.
\end{gather*}
Both  $\|\tilde\psi\|_\infty$ and $ \Lip_x(\tilde\psi)$ are bounded from above  by $\Lip(\phi)\Lip(\Gamma)( 1 + \diam(\Omega_{AS}))$. 

As $\tilde\psi = (\phi-\bar\phi_\Lambda) \circ \gamma_x $ is Lipschitz continuous, it is differentiable almost everywhere,  with a derivative in the distribution sense in $L^\infty(\tilde D_x( \tau,\epsilon))$. 

\medskip
{\it Item \ref{Item:PositiveLivsicChecking_1}.} Assume $\tilde z(T) \in \partial^+ \tilde D_x(\tau,\epsilon)$.  Define a primitive of $\tilde \psi(\cdot, w)$ for $w$ fixed by
\[
\forall\, t\in(-2\epsilon,2\tau), \ \forall\, w\in B_x(2\epsilon), \ F(t,w) := -\int_t^{{\tilde\tau_{x,y}}(w)} \! \tilde\psi(s,u) \, ds.
\]
For every $t \in[-\tau,2\tau]$ and for almost everywhere $u \in B_x(2\epsilon)$
\[
\frac{\partial F}{\partial t}(t,w)  = \tilde\psi(t,w), \quad \frac{\partial F}{\partial w}(t,w) = -\nabla {\tilde\tau_{x,y}}(w) \tilde \psi(t,w) - \int_t^{{\tilde\tau_{x,y}}(w)} \! \frac{\partial \tilde\psi}{\partial w}(s,w) \, ds.
\]
Define
\begin{gather*}
\|DF\|_x := \Big\| \frac{\partial F}{\partial t}\Big\|_x+ \Big\| \frac{\partial F}{\partial w} \Big\|_x, \quad \Big\| \frac{\partial F}{\partial w} \Big\|_x :=\sup_{(t,w) \in \tilde D(\tau,\epsilon)} \Big\| \frac{\partial F}{\partial w}(t,w) \Big\|_x.
\end{gather*}
Using the estimate  ${\tilde\tau_{x,y}} \leq 2\tau$, $-\tau \leq t \leq 2\tau$ and $\|\nabla {\tilde\tau_{x,y}} \|_x \leq 1$
\begin{gather*}
\Big\| \frac{\partial F}{\partial t}\Big\|_x \leq  \| \tilde\psi \|_\infty \leq \Lip(\phi)  \diam(\Omega_{AS}), \\
\Big\| \frac{\partial \tilde \psi}{\partial w} \Big\|_x \leq \Lip(\phi) \Lip(\Gamma), \quad
\Big\| \frac{\partial F}{\partial w} \Big\|_x  \leq \|\tilde\psi\|_\infty + 3\tau \Big\| \frac{\partial \tilde \psi}{\partial w} \Big\|_x, \\
\|DF\|_x \leq 3 \Lip(\phi)(\tau \Lip(\Gamma)+\diam(\Omega_{AS})).
\end{gather*}
Let
\begin{gather*}
C_1 := 6  \Lip(\phi)\Lip(\Gamma)(\tau \Lip(\Gamma)+\diam(\Omega_{AS})).
\end{gather*}
Then for every $C \geq C_1$,
\begin{gather}
\|DF\|_x\leq \frac{C}{2 \Lip(\Gamma)}. \label{Equation:PositiveLivsicChecking_4}
\end{gather}
The key identity is  to replace $\tilde \psi \circ  \tilde z$ by  the derivative with respect to $t$ of a function of the form $F \circ \tilde z$. We have
\begin{align}
\tilde\psi(r(s), q(s)) &= \frac{\partial F}{\partial t}(\tilde z(s)) \notag \\
&= \frac{\partial F}{\partial t} \big(\tilde z(s)) (1-r'(s) \big) - \frac{\partial F}{\partial w}(\tilde z(s)) q'(s) + \frac{d}{ds}\big(F(\tilde z(s)) \big). \label{Equation:PositiveLivsicChecking_3}
\end{align}
Using Cauchy-Schwarz inequality one obtains,
\begin{gather}
e_1 := (1,0),  \quad \|e_1-\tilde z'(s)\|^2=(1-r'(s))^2+q'(s)^2, \notag \\
\Big| \frac{\partial F}{\partial t} \big(\tilde z(s)) (1-r'(s) \big) - \frac{\partial F}{\partial w} (\tilde z(s)) q'(s) \Big|  \leq \|DF\|_x  \, \|e_1-\tilde z'(s) \|_x. \label{Equation:PositiveLivsicChecking_2}
\end{gather}
We obtain the main estimate: using \eqref{Equation:PositiveLivsicChecking_4},  \eqref{Equation:PositiveLivsicChecking_3} and \eqref{Equation:PositiveLivsicChecking_2}, one has
\begin{gather}
\tilde\psi(\tilde z(s))  
\geq \frac{d}{ds}\big(F(\tilde z(s)) \big) - \frac{C}{2 \Lip(\Gamma)}\|e_1-\tilde z'(s)\|_x. \label{Equation:PositiveLivsicChecking_2}
\end{gather}
Integrating over $[0,T]$, we have
\begin{align}
\int_0^T \! \|e_1-\tilde z'(s) \|_x \, ds &\geq \Big\| \int_0^T \! (e_1-\tilde z'(s)) \, ds \Big\|_x \notag \\
&\geq \| \big(T-r(T)+r(0),-q(T)+q(0) \big)\|_x \label{Equation:PositiveLivsicChecking_7} \\
&\geq  \|q(T)-q(0)\|_x. \label{Equation:PositiveLivsicChecking_5}
\end{align}
Using $\| \nabla {\tilde\tau_{x,y}} \|_x \leq 1$, we have
\begin{gather*}
|{\tilde\tau_{x,y}}(q(T))-{\tilde\tau_{x,y}}(q(0))| \leq  \|q(T)-q(0)\|_x, \\
\|q(T)-q(0)\|_x \geq \frac{1}{\sqrt{2}} \| \big({\tilde\tau_{x,y}}(q(T))-{\tilde\tau_{x,y}}(q(0)), q(T)-q(0) \big) \|_x.
\end{gather*}
By definition of  $(\tilde r_x ,\tilde q_x)$ and $(\tilde r_y ,\tilde q_y)$, by definition of the local Poincar\'e section and the local return time ${\tilde\tau_{x,y}}$, we have
\begin{gather*}
\begin{cases}
\gamma_y(0,f_{x,y}(\tilde q_x)) = \gamma_x({\tilde\tau_{x,y}}(q(0)),q(0)), \\
\gamma_y(0,\tilde q_y) = \gamma_x({\tilde\tau_{x,y}}(q(T)),q(T)) \in \Sigma_y,
\end{cases} \label{Equation:PositiveLivsicChecking_1}
\end{gather*}
and therefore
\begin{gather}
 \| \big({\tilde\tau_{x,y}}(q(T))-{\tilde\tau_{x,y}}(q(0)), q(T)-q(0) \big) \|_x \geq \frac{1}{\Lip(\Gamma)^{2}} \| f_{x,y}(\tilde q_x) - \tilde q_y \|_y, \notag \\
 \|q(T)-q(0)\|_x \geq \frac{1}{\sqrt{2}\Lip(\Gamma)^{2}} \| f_{x,y}(\tilde q_x) - \tilde q_y \|_y. \label{Equation:PositiveLivsicChecking_6}
\end{gather}
Using the local conjugacy to the constant flow $e_1$ 
\begin{align*}
z(T) &= \gamma_y(\tilde r_y ,\tilde q_y) = f^{\tilde r_y} \circ \gamma_y(0,\tilde q_y) \\
&=  f^{\tilde r_y} \circ  \gamma_x({\tilde\tau_{x,y}}(q(T)),q(T)) = \gamma_x(\tilde r_y + {\tilde\tau_{x,y}}(q(T)),q(T)) \\
&= \gamma_x(r(T),q(T)),
\end{align*}
on obtains
\begin{gather*}
r(T) = \tilde r_y + {\tilde\tau_{x,y}}(q(T)), \\
\| V \circ z(s) - z'(s) \| \geq \Lip(\Gamma)^{-1} \| e_1 - \tilde z'(s) \|_x.
\end{gather*}
Then using \eqref{Equation:PositiveLivsicChecking_2}, \eqref{Equation:PositiveLivsicChecking_5}, \eqref{Equation:PositiveLivsicChecking_6}, one has
\begin{align*}
\mathcal{A}_{\phi,C}(z) &= \int_0^T \! \Big( \tilde\psi \circ \tilde z(s) + C\| V \circ z(s) - z'(s) \| \Big) \, ds \\
&\geq \int_0^T \! \Big( \tilde\psi \circ \tilde z(s) + \frac{C}{\Lip(\Gamma)} \| e_1 - \tilde z'(s) \|_x \Big) \, ds \\
&\geq  \int_0^{T} \! \frac{d}{ds}\big(F(\tilde z(s)) \big) \, ds + \frac{C}{2 \Lip(\Gamma)} \|q(T)-q(0)\|_x \\
&\geq F(\tilde z(T)) - F(\tilde z(0)) + \frac{C}{2\sqrt{2} \, \Lip(\Gamma)^3} \| f_{x,y}(\tilde q_x) - \tilde q_y \|_y.
\end{align*}
Using  the identity $\gamma_x(s+{\tilde\tau_{x,y}}(q(T)),q(T)) = \gamma_y(s,\tilde q_y)$, we obtain
\begin{align*}
F(\tilde z(0)) &= -\int_{r(0)}^{{\tilde\tau_{x,y}}(q(0))} \! \tilde\psi(s,q(0)) \,ds \\
&= -\int_0^{{\tilde\tau_{x,y}}(\tilde q_x)} \tilde\psi (s,\tilde q_x) \, ds + \int_0^{\tilde r_x} \! (\phi -\bar\phi_\Lambda) \circ \gamma_x(s,\tilde q_x) \, ds  = -\Phi_{x,y} + \Psi_x,\\
F(\tilde z(T)) &= - \int_{r(T)}^{{\tilde\tau_{x,y}}(q(T))} \! \tilde\psi(s,q(T)) \, ds = \int_0^{r(T)-{\tilde\tau_{x,y}}(q(T))} \tilde\psi(s+{\tilde\tau_{x,y}}(q(T)),q(T)) \,ds = \Psi_y.
\end{align*}

\medskip
{\it Item \ref{Item:PositiveLivsicChecking_2}.} Assume $\tilde z(T) \in \partial^- \tilde D_x(\tau,\epsilon)$.   Define a different primitive of $\tilde\psi(\cdot, w)$ by
\[
\forall\, t\in(-2\epsilon,2\tau), \ \forall\, w\in B_x(2\epsilon), \ F(t,w) := \int_0^{t} \! \tilde\psi(s,w) \, ds.
\]
Then as above we have almost everywhere
\begin{gather*}
\frac{\partial F}{\partial t}(t,w)  = \tilde\psi(t,w), \quad \frac{\partial F}{\partial w}(t,w) =  \int_0^{t} \! \frac{\partial \tilde\psi}{\partial w}(s,w) \, ds,  \\
\Big\| \frac{\partial F}{\partial t}\Big\|_x \leq  \| \tilde\psi \|_\infty, \quad \Big\| \frac{\partial F}{\partial w} \Big\|_x \leq 2\tau \Big\| \frac{\partial \tilde \psi}{\partial w} \Big\|_x,
\end{gather*}
and for every $C\geq C_1$,
\begin{gather}
\|DF\|_x \leq 2  \Lip(\phi)(\tau \Lip(\Gamma) + \diam(\Omega_{AS})) \leq \frac{C}{2 \Lip(\Gamma)}. \label{Equation:PositiveLivsicChecking_8}
\end{gather}
Using  \eqref{Equation:PositiveLivsicChecking_3}, one obtains
\begin{align*}
\tilde\psi(r(s), q(s)) &= \frac{\partial F}{\partial t}(\tilde z(s)) \\
&= \frac{\partial F}{\partial t} \big(\tilde z(s)) (1-r'(s) \big) - \frac{\partial F}{\partial w}(\tilde z(s)) q'(s) + \frac{d}{ds}\big(F(\tilde z(s)) \big).
\end{align*}
Using \eqref{Equation:PositiveLivsicChecking_7}, one obtains
\begin{multline}
\mathcal{A}_{\phi,C}(z)  \geq F(\tilde z(T)) - F(\tilde z(0)) \\
+ \frac{C}{2 \Lip(\Gamma)}  \| \big(T-r(T)+r(0),-q(T)+q(0) \big)\|_x. \label{Equation:PositiveLivsicChecking_11}
\end{multline}
Either $\tilde z(T) \in [-2\epsilon,\tau] \times \partial B_x(2 \epsilon)$,
\begin{gather*}
\| \big(T-r(T)+r(0),q(T)-q(0) \big)\|_x  \geq \|q(T)-q(0)\|_x \geq \epsilon, 
\end{gather*}
or $\tilde z(T) \in \{-2\epsilon\} \times \overline{B_x(2\epsilon)}$, $r(T)=-2\epsilon$,
\begin{multline*}
\| \big(T-r(T)+r(0),q(T)-q(0) \big)\|_x  \\ 
\geq |T-r(T)+r(0)| = (T+2\epsilon-r(0)) \geq \epsilon. 
\end{multline*}
In both cases, using \eqref{Equation:PositiveLivsicChecking_11}, one obtains
\begin{gather}
\mathcal{A}_{\phi,C}(z)  \geq  \frac{C \epsilon}{2 \Lip(\Gamma)} - \big| F(\tilde z(T)) - F(\tilde z(0)) \big|. \label{Equation:PositiveLivsicChecking_9}
\end{gather}
Using $\epsilon < \frac{\tau}{2}$ one obtains
\begin{gather*}
\| \tilde z(T)-\tilde z(0) \|_x \leq \sqrt{(4\epsilon)^2+(3\tau)^2} \leq 4 \, \tau.
\end{gather*}
Using \eqref{Equation:PositiveLivsicChecking_8} one obtains
\begin{gather*}
| F(\tilde z(T)) - F(\tilde z(0)) |   \leq 8 \, \tau \Lip(\phi)(\tau\Lip(\Gamma)+\diam(\Omega_{AS})), \end{gather*}
For every $C \geq \frac{16}{3} \frac{\tau}{\epsilon}C_1$
\begin{gather}
| F(\tilde z(T)) - F(\tilde z(0)) | \leq \frac{C \epsilon}{4\Lip(\Gamma)}. \label{Equation:PositiveLivsicChecking_10}
\end{gather}
Let for every $\mathcal{A}>0$,
\begin{gather*}
C_2(\mathcal{A}) := \max \Big( \frac{16}{3} \frac{\tau}{\epsilon}C_1, \frac{4}{\epsilon}\Lip(\Gamma) \mathcal{A} \Big).
\end{gather*} 
Then for every $C \geq C_2(\mathcal{A})$, using   \eqref{Equation:PositiveLivsicChecking_9}, \eqref{Equation:PositiveLivsicChecking_10},  one obtains
\begin{gather*}
\mathcal{A}_{\phi,C}(z)  \geq \frac{C\epsilon}{4 \Lip(\Gamma)} \geq \mathcal{A}.
\end{gather*}

{\it Item \ref{Item:PositiveLivsicChecking_3}.} We bound from below the penalized action as in item \ref{Item:PositiveLivsicChecking_2}. The path may not exit the trap box and we are left to  the estimate \eqref{Equation:PositiveLivsicChecking_11}. We then obtain
\begin{align*}
\mathcal{A}_{\phi,C}(z)  &\geq - | F(\tilde z(T)) - F(\tilde z(0)) | \\
&\geq -8\tau \Lip(\phi)(\tau\Lip(\Gamma)+\diam(\Omega_{AS})) = \mathcal{A}_1.\qedhere
\end{align*}
\end{proof}

We now prove the validity of the positive Liv\v{s}ic criterion (Definition \ref{Definition:PositiveLivcicCriteria}) in the context of locally maximal compact hyperbolic invariant set. 

We first redefine the notion of pseudo orbit in the context of flows. In words a pseudo orbit is a concatenation of paths of the first case (Lemma \ref{Lemma:PositiveLivsicChecking}) that start in a local ball $U_{x_i}$, end at a distinct local ball $U_{x_{i+1}}$, and at the same time time exit the trap box at the forward boundary $\partial^+ D_{x_i}(\tau,\epsilon)$.

\begin{definition} \label{Definition:PeriodicPseudoOrbitFlow}
Let $\Gamma$ be a family of adapted local flow boxes as defined in Definition~\ref{Definition:AdaptedLocalFlowBoxes}, $\epsilon < \frac{1}{2} \tau$, and $(D_x(\tau,\epsilon))_{x \in\Lambda}$, $(U_x(\epsilon))_{x \in \Lambda}$ be respectively the family of local trap boxes of size $(\tau,\epsilon)$ and local balls of size $\epsilon$ as defined in Notation~\ref{Definition:NotationsReducedPositiveLivsic}. Let   $\Lambda_* \subseteq \Lambda$ be a finite set, and $\Omega \supseteq \Lambda$ be an open set such that $\Omega \subseteq \bigcup_{x \in \Lambda_*}U_x(\epsilon)$.  A {\it $(\Lambda_*,\tau,\epsilon)$-pseudo orbit} is a  piecewise $C^1$ continuous  path $z:[0,T] \to \Omega$ such that there exist $N\geq1$, points $x_0,\ldots,x_N$ in $\Lambda_*$ and an increasing sequence of times  $T_0=0 < T_1< \cdots < T_N=T$ such that
\begin{enumerate}
\item $\forall\, i \in \llbracket 0,N\rrbracket, \ z(T_i) \in U_{x_i}(\epsilon)$,
\item $\forall\, i \in \llbracket 0,N-1\rrbracket, \ \forall\, t \in [T_{i},T_{i+1}], \ z(t) \in D_{x_{i}}(\tau,\epsilon)$,
\item $\forall\, i \in \llbracket 0,N-1\rrbracket, \ z(T_{i+1}) \in \partial^+D_{x_i}(\tau,\epsilon)$.
\end{enumerate}
The times $(T_i)_{i=0}^N$ are called {\it cutting times}, and the points $(x_i)_{i=0}^N$ the {\it cutting points}. A periodic $(\Lambda_*,\epsilon)$-pseudo orbit satisfies in addition
\begin{enumerate}
\addtocounter{enumi}{3}
\item $x_0=x_N$.
\end{enumerate}
\end{definition}

We also recall a more precise version of the Anosov shadowing Lemma  \cite[Theorem 2.1]{SuThieullenYu2021} for discrete  pseudo orbits between Poincar\'e sections. In that version the total sum of the distance between the pseudo orbit and the periodic orbit is bounded by the total sum of the  errors made during the shadowing approximation. The standard Anosov shadowing lemma gives a bound of the supremum of the distances with respect to the supremum of the errors  but takes into account the number errors.

\begin{lemma}[Adapted Anosov shadowing Lemma,  {\cite[Theorem 2.1]{SuThieullenYu2021}}] \label{Lemma:AdaptedAnosovShadowingLemma}
Let $\Gamma$ be a family of adapted local flow boxes, $\Lambda_* \subseteq \Lambda$ be  a finite set, $(x_i)_{i=0}^N$ be a periodic sequence,  $x_N=x_0$,  of points of $\Gamma$ so that   $(x_{i-1}, x_i)$ is forward admissible as  in item \ref{Item:AdaptedLocalFlowBoxes_7bis} in Definition \ref{Definition:AdaptedLocalFlowBoxes}. Let $B_i(\rho) = B_{x_i}(\rho)$ be the  adapted balls and $f_i = f_{x_i,x_{i+1}} : B_i(\rho) \to B_{i+1}(1)$ be the local Poincar\'e map. Then there exists a constant $K_\Lambda \geq1$ such that, for every periodic pseudo orbits $(q_i)_{i=0}^N$  in the sense 
\[
\forall\,  i \in \llbracket 0, N-1 \rrbracket, \ q_i\in B_i(\rho/2), \ f_i(q_i), q_{i+1}  \in B_{i+1}(\rho/2),
\]
there exists a periodic orbit $(p_i)_{i=0}^N$, $p_N=p_0$ such that
\[
\begin{cases}
\forall\, i \in \llbracket 0,N-1 \rrbracket$, $f_i(p_i)=p_{i+1}, \\
\sum_{i=0}^{N-1} \|q_i-p_i\|_i \leq  K_\Lambda \sum_{i=0}^{N-1} \|f_i(q_i) - q_{i+1} \|_{i+1}.
\end{cases}
\]
\end{lemma}

The proof of Theorem \ref{Theorem:ContinuousPositiveCriteria} is a consequence of the following three lemmas. In both lemmas we consider pseudo orbit that are concatenation of paths of the first case of Lemma \ref{Lemma:PositiveLivsicChecking}. We recall some notations. Let  $(\tau,\epsilon)\in \mathbb{R}_+^2$, $T>0$, $\Lambda_* \subset \Lambda$  be a finite set, $\epsilon_{AS} < \epsilon$, $\Omega_{AS} \subset \bigcup_{x\in \Lambda_*}U_x(\epsilon)$ be a neighborhood of $\Lambda$, and $N_* := \Card(\Lambda_*)$ as in Notation \ref{Definition:NotationsReducedPositiveLivsic}.  Let $z : [0,T] \to \Omega_{AS}$ be a $(\Lambda_*, \tau,\epsilon)$-pseudo orbit satisfying  Definition \ref{Definition:PeriodicPseudoOrbitFlow}.

In the first lemma the pseudo orbit is periodic.    We use essentially item~\ref{Item:PositiveLivsicChecking_1} of Lemma \ref{Lemma:PositiveLivsicChecking} and the  adapted Anosov shadowing Lemma \ref{Lemma:AdaptedAnosovShadowingLemma}. The proof consists  in writing the penalized action $\mathcal{A}_{\phi,C}(z)$ of a pseudo orbit $z : [0,T] \to \Omega$  as a sum  of a coboundary $\Psi_{i+1} - \Psi_i$ and  a penalized discrete action $\Phi_i(\tau_i,q_i)$ computed on a true orbit $t \in [0,\tau_i] \mapsto f^t(z_i) \in \Omega$ between two Poincar\'e sections $\Sigma_{x_i}, \Sigma_{x_{i+1}}$ at successive cutting points $x_i$. 

\begin{lemma} \label{Lemma:PeriodicPseudoBirkhoffIntegral}
Assume $z:[0,T] \to \Omega_{AS}$ is a periodic pseudo orbit. Define 
\begin{gather*}
C_3 := 4\sqrt{2} \, \Lip(\Gamma)^3 \Lip(\phi)(  \tau \Lip(\Gamma) + \diam(\Omega_{AS})) K_\Lambda, \\
\mathcal{A}_2^* := 2\epsilon \Lip(\phi) \diam(\Omega_{AS}).
\end{gather*}
Then for every $C \geq C_3$,
\[
\mathcal{A}_{\phi,C}(z) \geq -  \mathcal{A}_2^*.
\]
\end{lemma}

\begin{proof}
Using the notations of Definition \ref{Definition:PeriodicPseudoOrbitFlow}, by definition of a periodic $(\Lambda_*,\epsilon)$-pseudo orbit of length $N$, there exist a sequence of cutting points $(x_i)_{i=0}^N$,  $x_i \in \Lambda_*$, $x_N=x_0$, there exists a sequence of cutting  times $(T_i)_{i=0}^N$, $T_0=0 < T_1 < \cdots < T_N$ such that $z_i:=z(T_i) \in U_{x_i}(\epsilon)$. Let $(r_i,q_i) := \gamma_{x_i}^{-1}(z_i)$ be  the box coordinates of $z_i$ in the chart $\gamma_{x_i}$, with $|r_i| < \epsilon$ and $q_i \in  B_{x_i}(\epsilon)$. Let $ \tau_i (q) := \tilde \tau_{x_i,x_{i+1}}(q)$ be the   return time between  the two Poincar\'e sections $\Sigma_{x_i}$ and $\Sigma_{x_{i+1}}$ passing through $x_i$ and $x_{i+1}$ respectively as in item \ref{Item:AdaptedLocalFlowBoxes_5} of Definition \ref{Definition:AdaptedLocalFlowBoxes}. Let $f_i(q) = f_{x_i,x_{i+1}}(q) : B_{x_i}(\rho) \to B_{x_{i+1}}(1)$  be the local Poincar\'e map as in item \ref{Item:AdaptedLocalFlowBoxes_6} in Definition \ref{Definition:AdaptedLocalFlowBoxes}. 

Using  item \ref{Item:PositiveLivsicChecking_1} of Lemma \ref{Lemma:PositiveLivsicChecking} and $\tilde C := {C}/{\sqrt{8} \, \Lip(\Gamma)^3}$, one obtains
\begin{align*}
\mathcal{A}_i &:= \int_{T_i}^{T_{i+1}} \big[ (\phi - \bar \phi_\Lambda) z(s)  +  C \| V \circ z(s) -z'(s) \| \big] \, ds \\
&\geq \Phi_i(\tau_i(q_i),q_i) +\Psi_{i+1} -\Psi_i  + \tilde C \| f_{i}( q_i) -  q_{i+1}\|_{i+1},
\end{align*}
where for every $t\in(-\tau,2\tau)$ and $q \in B_{x_i}(\epsilon)$,
\begin{gather*}
\Phi_i (t,q) := \int_0^{t} \!  (\phi-\bar\phi_\Lambda) \circ \gamma_{x_i}(s,q) \, ds, \quad 
\Psi_i := \Phi_i(r_i,q_i).
\end{gather*}
Notice that, though $x_0=x_N$, $\Phi_0(t,q) = \Phi_N(t,q)$ for every $(t,q)$, but $z(0) \neq z(T)$ and $\Psi_0 \neq \Psi_N$. 

Then $(q_i)_{i=0}^N$ is a periodic pseudo orbit as in  Lemma \ref{Lemma:AdaptedAnosovShadowingLemma}. There exists a periodic orbit $(p_i)_{i=0}^N$, that is a sequence of points satisfying
\begin{gather*}
p_N=p_0, \ \forall\, i \in \llbracket 0, N-1\rrbracket, \ f_i(p_i)= p_{i+1}, \\
\sum_{i=0}^{N-1} \|q_i-p_i\|_i \leq  K_\Lambda \sum_{i=0}^{N-1}  \|f_{i}(q_{i}) - q_{i+1} \|_{i+1}.
\end{gather*}
Using the estimate $\|\nabla \tau_i\|_i \leq 1$ from item \ref{Item:AdaptedLocalFlowBoxes_5} of Definition \ref{Definition:AdaptedLocalFlowBoxes}, the Lipschitz constants of $\Phi_i$ and $\Psi_i$ can be computed in the following way
\begin{gather*}
\Big\| \frac{\partial \Phi_i}{\partial t}\Big\|_i \leq \Lip(\phi) \diam(\Omega_{AS}), \quad \Big\| \frac{\partial \Phi_i}{\partial q}\Big\|_i \leq 2\tau \Lip(\phi)\Lip(\Gamma), \\
\begin{split}
\| \Phi_i(\tau_i(q),q) -\Phi_i(\tau_i(p),p) \|_i &\leq  \Big[ \Big\| \frac{\partial \Phi_i}{\partial t}\Big\|_i+ \Big\| \frac{\partial \Phi_i}{\partial q}\Big\|_i  \Big]\, \|q-p \|_i \\
&\leq  2\Lip(\phi) \big( \tau \Lip(\Gamma) + \diam(\Omega_{AS}) \big)  \|q-p \|_i .
\end{split}
\end{gather*}
Summing over $i$ we obtain
\begin{align*}
\sum_{i=0}^{N-1} &\big[  \Phi_i(\tau_i(q_i),q_i) +\Psi_{i+1}(q_{i+1})-\Psi_i(q_i) \big] \\ 
&= \sum_{i=0}^{N-1}  \Phi_i(\tau_i(p_i),p_i) + \sum_{i=0}^{N-1} \big[ \Phi_i(\tau_i(q_i),q_i) - \Phi_i(\tau_i(p_i),p_i) \big] + \Psi_N - \Psi_0.
\end{align*}
As $(p_i)_{i=0}^N$ is a (discrete) periodic orbit, denoting 
\[
w_i := \gamma_{x_i,0}(p_i), \ \ S := \sum_{i=0}^{N-1}\tau_i(p_i),
\] 
one obtains $f^{\tau_i(p_i)}(w_i)=w_{i+1}$. Thus  $f^{S}(w_0)=w_N=w_0$, $(f^t(w_0))_{t \in\mathbb{R}}$ is a (continuous) periodic orbit of $\Lambda$ of period $S$, and
\[
\sum_{i=0}^{N-1}   \Phi_i(\tau_i(p_i),p_i) =\int_0^{S} (\phi-\bar \phi_\Lambda) \circ f^s(w_0) \, ds  \geq0.
\]
Moreover 
\begin{multline*}
\Big| \sum_{i=0}^{N-1} \big[ \Phi_i(\tau_i(q_i),q_i) - \Phi_i(\tau_i(p_i),p_i) \big] \Big| \\ 
\leq  2\Lip(\phi) \big( \tau \Lip(\Gamma) + \diam(\Omega_{AS}) \big)   \sum_{i=0}^{N-1} \| q_i - p_i \|,
\end{multline*}
and
\begin{gather*}
|\Psi_N - \Psi_0| \leq  |\Psi_N| + |\Psi_0| \leq 2\epsilon \Lip(\phi) \diam(\Omega_{AS}) = \mathcal{A}_2^*.
\end{gather*}
Using the shadowing Lemma \ref{Lemma:AdaptedAnosovShadowingLemma} and $\tilde C =  {C}/({2\sqrt{2} \, \Lip(\Gamma)^3})$, one obtains
\begin{align*}
\sum_{i=0}^{N-1} \mathcal{A}_i &\geq  \Psi_N(p_N)-\Psi_0(p_0) \\
&\quad + \Big( \tilde C -   2\Lip(\phi)(  \tau \Lip(\Gamma) + \diam(\Omega_{AS})) K_\Lambda \Big) \\
&\quad\quad\quad \times \sum_{i=0}^{N-1}\| f_{i}( q_i) -  q_{i+1}\|_{i+1} \geq -\mathcal{A}_2^*. \qedhere
\end{align*}
\end{proof}

In the second lemma  the pseudo orbit may not be periodic. Let $(x_i)_{i=0}^N$ be the cutting points of that pseudo orbit. We consider $\Lambda_*$ as the set of vertices of some graph where $(x,y)$ is an ordered edge if $(x,y)$ is $\Gamma$ admissible. Then $(x_i)_{i=0}^N$ can be seen as a non injective path in that graph. We show that it can be decomposed into a  self-avoiding path $(x_{i_0}, x_{i_1}, \ldots, x_{i_r})$, $i_0=0, i_r =N$, that connects periodic cycles $(x_{i_{k}}, x_{i_{k}+1}, \ldots, x_{i_{k+1}-1})$ pinned at $x_{i_{k+1}-1}=x_{i_k}$. The cardinality of the self-avoiding path, or the cardinality of the periodic cycles is less that the cardinality of $\Lambda_*$ (see Figure \ref{Figure:SelfAvoidingPath}).

\begin{figure}[hbt]
\centering
\includegraphics[width=1\textwidth]{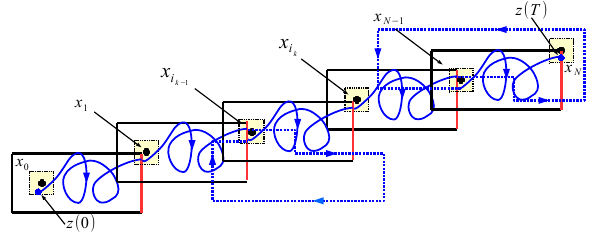}
\caption{A periodic pseudo orbits (in dash blue) $(x_{i_{k-1}}, x_{i_{k-1}+1}, \ldots, x_{i_k -1})$ where $x_{i_{k-1}} = x_{i_k-1}$ followed by an edge of the  self-avoiding path $(x_{i_k-1}, x_{i_k})$ (in blue). The figure corresponds to $i_r=N$ and $i_{r-1} = N-1$ since the last periodic pseudo orbit is followed by an edge  of the self-avoiding path $(x_{N-1},x_N)$.} \label{Figure:SelfAvoidingPath}
\end{figure}

\begin{lemma} \label{Lemma:PseudoPeriodicPath}
We do not assume that the pseudo orbit $z:[0,T] \to \Omega_{AS}$ is periodic. Let  $x_0, \ldots,x_N$ the cutting points (Definition~\ref{Definition:PeriodicPseudoOrbitFlow}). Then there exist $r \in \llbracket 1, N_* \rrbracket$ and $i_0 = 0 < i_1 < \cdots < i_r=N$,  such that
\begin{enumerate}
\item $x_{i_0},x_{i_1},\ldots,x_{i_{r-1}}$ are distinct,
\item if $k \in \llbracket 1,r-1\rrbracket$ then
\begin{enumerate}
\item either $i_k = i_{k-1}+1$, then $(x_{i_{k-1}}, x_{i_k})$ is an edge of the self-avoiding path,
\item or $i_k \geq i_{k-1}+2$, then $(x_{i_{k-1}}, x_{i_{k-1}+1}, \ldots, x_{i_k -1})$, is a periodic cycle pinned at $x_{i_k-1}=x_{i_{k-1}}$, followed by an edge $(x_{i_k-1},x_{i_k})$  of the self-avoiding path,
\end{enumerate} 
\item if $i_r \geq i_{r-1}+2$, then $(x_{i_{r-1}}, x_{i_{r-1}+1}, \ldots, x_{i_r})$ with $x_{i_r}=x_{i_{r-1}}$ is the last periodic cycle.
\end{enumerate}
\end{lemma}

\begin{proof}
By definition of a $(\Lambda_*,\tau,\epsilon)$-pseudo orbit, we have
\[
z(T_i) \in U_{x_i}(\epsilon), \ \forall\, t\in[T_i,T_{i+1}], \ z(t) \in D_{x_i}(\tau,\epsilon), \ z(T_{i+1}) \in \partial^+D_{x_i}(\tau,\epsilon).
\]
In particular $U_{x_i}(\epsilon) \bigcap U_{x_{i+1}}(\epsilon) = \emptyset$ and $x_{i} \neq x_{i+1}$. Assume by induction that we have constructed indices $i_0, \ldots,i_{k}$ in $\llbracket 0,N-1 \rrbracket$ satisfying the 3 items and  such that 
\[
\forall\, i \geq i_k, \ x_i \not\in \{ x_{i_0}, \ldots, x_{i_{k-1}} \}.
\]
Let be $I:=\{ i \geq i_k+1 : x_i = x_{i_k} \}$. If $I = \emptyset$, then $i_{k+1}=i_k+1$. If $I \not= \emptyset$, $j = \max(I)$, then $j > i_k+1$. Either $j<N$, then $i_{k+1}=j+1$, $x_{i_{k+1}-1}=x_j=x_{i_k}$, or $j=N$, then $k=r-1$, $i_r=N$,  $x_{i_r}=x_{i_{r-1}}$, and the induction step is proved.
\end{proof}

In the third lemma we bound from below  the penalized action $\mathcal{A}_{\phi,C}(z)$ of any pseudo orbit. We concatenate periodic pseudo orbits and a finite number of trap paths. The number of trap boxes $N_*$ that cover $\Lambda$ is involved in the computation. 

\begin{lemma} \label{lemma:FinitenessPseudoBirkhoffIntegral}
Let  $C\geq C_3$ as in Lemma \ref{Lemma:PeriodicPseudoBirkhoffIntegral} and
\[
\mathcal{A}_{3}^* := (\mathcal{A}_1^*+\mathcal{A}_2^*)N_*.
\]  
Then for every $(\Lambda_*,\tau,\epsilon)$-pseudo orbit $z:[0,T] \to \Omega_{AS}$, we have
\[
\mathcal{A}_{\phi,C}(z) \geq -\mathcal{A}_{3}^*.
\]
\end{lemma}

\begin{proof}
Using the notations of Lemma \ref{Lemma:PseudoPeriodicPath}, we decompose the pseudo orbit $z : [0,T] \to \Omega_{AS}$ into $r$ pseudo orbits $z_k : [T_{i_k},T_{i_{k+1}}] \to \Omega_{AS}$  where $k\in \llbracket 0,r-1 \rrbracket$.

Either ${i_{k+1}}={i_k}+1$. Then  $z_k([T_{i_k},T_{i_{k+1}}])$ is included into the trap box $D_{x_{i_k}}(\tau,\epsilon)$ and item \ref{Item:PositiveLivsicChecking_3} of Lemma \ref{Lemma:PositiveLivsicChecking} implies
\[
\mathcal{A}_{\phi,C}(z_k) \geq -\mathcal{A}_1^*.
\]
Or ${i_{k+1}}\geq {i_k}+2$.  Then both $ z(T_{i_k}), z(T_{i_{k+1}-1})$ belong to $U_{x_{i_k}}(\epsilon)$. The path $z_k : [T_{i_k},T_{i_{k+1}}] \to \Omega_{AS}$ is the concatenation of two paths (one path if $k=r-1$): the periodic cycle $z'_k : [T_{i_k},T_{i_{k+1}-1}] \to \Omega_{AS}$ followed by the trap path $z''_k : [T_{i_{k+1}-1},T_{i_{k+1}}] \to \Omega_{AS}$ included in the trap box $D_{x_{i_k}}(\tau,\epsilon)$  that is part of the self-avoiding path.  Lemma \ref{Lemma:PeriodicPseudoBirkhoffIntegral} and item \ref{Item:PositiveLivsicChecking_3} of Lemma \ref{Lemma:PositiveLivsicChecking} imply
\[
\mathcal{A}_{\phi,C}(z'_k) \geq - \mathcal{A}_2^*, \quad \mathcal{A}_{\phi,C}(z''_k) \geq -\mathcal{A}_1^*.
\]
We conclude the proof using  the fact that the path $z$ is the concatenation of $N_*$ paths $z_k$.
\end{proof}

\begin{proof}[\bf Proof of Theorem \ref{Theorem:ContinuousPositiveCriteria}]
Let $(\tau,\epsilon)\in \mathbb{R}_+^2$, $\Lambda_* \subset \Lambda$ a finite set, $\epsilon_{AS} < \epsilon$, $\Omega_{AS} \subset \bigcup_{x\in \Lambda_*}U_x(\epsilon)$, and $N_* := \Card(\Lambda_*)$ as in Notation \ref{Definition:NotationsReducedPositiveLivsic}. Define $\mathcal{A}_{3}^*$ as in Lemma \ref{lemma:FinitenessPseudoBirkhoffIntegral}, and $C_3$ as in Lemma \ref{Lemma:PeriodicPseudoBirkhoffIntegral}. Let
\begin{equation*}
\label{equation:Explicitepenalized Constant}
C_4 := \max(C_3,C_2(\mathcal{A}_3^*)),
\end{equation*}
where $C_2(\mathcal{A})$ has been defined in  Lemma \ref{Lemma:PositiveLivsicChecking}. 

Let $z:[0,T] \to \Omega_{AS}$ be a  piecewise $C^1$ continuous  path. We first decompose $z$ into $N\geq1$ subpaths $z_k:[T_{k-1},T_{k}] \to \Omega_{AS}$, where 
\begin{gather*}
T_0=0 < T_1 < \cdots <T_{N-1} \leq T_N=T,
\end{gather*} 
and $(x_0,x_1, \ldots, x_{N-1)}$ are  points in $\Lambda_*$ satisfying
\begin{itemize}
\item $\forall\, k \in \llbracket 0, N-1 \rrbracket, \ z(T_{k}) \in U_{x_{k}}(\epsilon)$, 
\item $\forall\, k \in \llbracket 1, N-1 \rrbracket, \ z(T_{k}) \in \partial D_{x_{k-1}}(\tau,\epsilon)$, 
\item $\forall\, k \in \llbracket 1, N \rrbracket, \ z([T_{k-1},T_{k}]) \subseteq \overline{D_{x_{k-1}}(\tau,\epsilon)}$.
\end{itemize}
Notice that $z(T_N)=z(T)$ may not belong to the boundary of a trap box and  $T_{N-1}=T_N$ if $z(T)$ belongs to the boundary of a trap box.  We then concatenate the subpaths $z_k$ into $r\geq1$ longer blocks $z_k : [T_{i_{k-1}},T_{i_{k}}] \to \Omega_{AS}$, where $i_0=0 < i_1 < \cdots < i_r=N$, such that each subpath $z_k$ is  one of the following 3 types:
\begin{itemize}
\item type I: $i_{k}=i_{k-1}+1$; the subpath  $z_k$ is an escaping path exiting at the backward boundary $z(T_{i_{k}}) \in \partial^- D_{x_{i_{k-1}}}(\tau,\epsilon)$, 
\item type II: $i_{k} \geq i_{k-1}+2$; the subpath $z_k$ is the concatenation of 2 paths,  a $(\Lambda_*,\tau,\epsilon)$-pseudo orbit $z'_k:[T_{i_{k-1}}, T_{i_k-1}]  \to  \Omega_{AS}$ where $z(T_{i_{k}-1}) \in \partial^+ D_{x_{i_{k}-2}}(\tau,\epsilon)$, followed by a second path $z''_k:[T_{i_k-1},T_{i_k}] \to  \Omega_{AS}$  consisting in an escaping path where $z(T_{i_k}) \in \partial^- D_{x_{i_k-1}}(\tau,\epsilon)$,
\item type III: $k=r$,   $z_{r} : [T_{N-1},T_N] \to \Omega_{AS}$ is a trap path.
\end{itemize}
Notice that type III may only happen as a terminal path $z_{r-1}$. 

If $z_k$ is of type I, if $C \geq C_2(\mathcal{A}_3^*)$, then item \ref{Item:PositiveLivsicChecking_2} of Lemma \ref{Lemma:PositiveLivsicChecking} implies that
\begin{gather*}
\mathcal{A}_{\phi,C}(z_k)\geq \mathcal{A}_3^* \geq 0.
\end{gather*}

If $z_k$ is of type II,   Lemma \ref{lemma:FinitenessPseudoBirkhoffIntegral} implies  $\mathcal{A}_{\phi,C}(z'_k) \geq -\mathcal{A}_{3}^*$ for the first path, item \ref{Item:PositiveLivsicChecking_2} of Lemma \ref{Lemma:PositiveLivsicChecking} implies   $\mathcal{A}_{\phi,C}(z''_k) \geq \mathcal{A}_3^*$ for the second path. Then we obtain the crucial estimate
\begin{gather*}
\mathcal{A}_{\phi,C}(z_k)  = \mathcal{A}_{\phi,C}(z'_k)  + \mathcal{A}_{\phi,C}(z''_k)  \geq0.
\end{gather*}
In conclusion the penalized action is non negative either for an escaping path, or for a pseudo orbit followed by an escaping path. We are left to the terminal path $\mathcal{A}_{\phi,C}(z_{r-1})  \geq   -\mathcal{A}_1^*$. The total penalized action is bounded from below by $\mathcal{A}_1^*$. To conclude we observe that $\Lip(\phi)$ can be factorized in the two constants $C_4$ and $\mathcal{A}_1^*$
\begin{gather*}
C_\Lambda := \frac{C_4}{\Lip(\phi)}, \quad \delta_\Lambda := \frac{\mathcal{A}_1^*}{\Lip(\phi)}. \qedhere
\end{gather*}
\end{proof}

%%%%%%%%%%%%%%%%%%%%%%
\section{The continuous Lax-Oleinik semigroup}
\label{Section:LaxOleinikSemigroup}
%%%%%%%%%%%%%%%%%%%%%%

We consider in this section a $C^1$ flow $(M,V,f)$, a compact connected invariant set $\Lambda$, and an open set $U$ containing $\Lambda$ of compact closure. We do not assume that $\Lambda$ is hyperbolic. We show that any Lipschitz observable satisfying the  positive Liv\v{s}ic criterion on $(U,\Lambda)$ admits a Lipschitz subaction on $U$ as in Definition~\ref{Definition:ContinuousErgodicOptimization}. The subaction is obtained as a fixed point of some nonlinear semigroup, called {\it Lax-Oleinik semigroup}. A similar definition  exists in  the context of weak KAM theory, see Fathi's monograph \cite{Fathi2016}. 

As $\Lambda$ is connected, we may choose $U$ connected.  We choose the following definition for the distance function between two points.

\begin{definition}
Let $U$ be a connected open set with compact closure. The {\it distance function between two points $p,q \in U$} is defined by
\begin{multline*}
d_U(p,q) =\inf \Big\{ \int_0^1 \|z'(s) \| \, ds : z : [0,1] \to U \ \text{is  continuous}, \\ 
\text{piecewise $C^1$}, \ z(0)=p \ \text{and} \  z(1)=q \Big\}.
\end{multline*}
Notice that $d_U(p,q)$ can be obtained by taking the infimum over paths defined on any interval $[0,T]$ instead of $[0,1]$.
\end{definition}

\begin{definition}[Lax-Oleinik Semigroup] \label{Definition:ContinuousLaxOleinikSemigroup}
Let  $\phi: U\to \mathbb{R}$ be a $C^0$ bounded function and  $C\geq0$ be a  constant. Assume $\phi$ satisfies the  positive Liv\v{s}ic criterion on $(U,\Lambda)$ with  penalized  constant $C$ as in Definition \ref{Definition:PositiveLivcicCriteria}.
\begin{enumerate} 
\item The {\it Lax-Oleinik semigroup on $(U,\Lambda)$ of generator $\phi$} is the nonlinear operator acting on bounded  functions $u : U \to \mathbb{R} $ defined for every $t > 0$ by, for every $q \in U$,
\begin{gather*}
T^t [u](q) :=\inf_{\substack{z:[-t,0] \to U \\ z(0)=q}}\Big\{u \circ z(-t)+\int_{-t}^{0} \big[\phi \circ z-\bar{\phi}_\Lambda+C \|V \circ z -z^{\prime}\| \big] \, ds\Big\} 
\end{gather*}
where the infimum is taken over the set of  piecewise $C^1$  continuous paths $z: [-t,0] \to U$ ending at $q$.

\item A {\it weak KAM solution} of the Lax-Oleinik semigroup is a bounded  function  $u : U \to \mathbb{R}$ solution of the equation
\[
\forall\, t > 0, \ T^{t}[u] = u.
\]
\end{enumerate}
\end{definition}

Notice that a weak KAM solution is a particular integrated subaction (item \ref{Label:ContinuousErgodicOptimization_2} of Definition \ref{Definition:ContinuousErgodicOptimization}). Indeed, taking $z(s) = f^s(p)$, one obtains $V \circ z(s) = z'(s)$ and
\begin{gather*}
u \circ f^t(p) = T^t[u](f^t(p) \leq u(p) + \int_0^t (\phi \circ z - \bar \phi_\Lambda) \circ f^s(p) \, ds.
\end{gather*}

\begin{theorem} \label{Theorem:ContinuousWeakKAMsolution}
Let $(M,V,f)$ be a $C^1$  flow, $\Lambda$ be a compact connected invariant set, and $U \supseteq \Lambda$ be a connected open set of compact closure. Let $\phi:U \to \mathbb{R}$ be a bounded Lipschitz continuous function and  $C\geq0$ be a  constant.  Assume $\phi$ satisfies the  positive Liv\v{s}ic criterion on  $(U,\Lambda)$ with penalized  constant $C$. Then there exists  a weak KAM solution  $u: U\to \mathbb{R}$ that is $C$-Lipschitz. Moreover $u$ is a Lipschitz continuous integrated subaction. 
\end{theorem}

The proof of Theorem \ref{Theorem:ContinuousWeakKAMsolution} requires the following two lemmas. We first extend the definition of a penalized action between two points. 

\begin{definition}
\label{Definition:ActionBetweenPoints}
Let $\phi : U \to \mathbb{R}$ be a bounded continuous function and  $C\geq0$ be a  constant. The {\it penalized  action of $\phi$  between two points $p,q \in U$} with a penalized constant $C$ and a time laps $t>0$ is the quantity
\[
\mathcal{A}_{\phi,C}^t(p,q) := \inf_{\substack{z:[0,t] \to U \\ z(0)=p, \, z(t) =q}} \int_0^t \! \big[\phi \circ z-\bar{\phi}_\Lambda+C \|V \circ z -z^{\prime}\| \big] \, ds,
\]
where the infimum is realized over the set of  piecewise $C^1$  continuous  paths starting at $p$ and ending at $q$. 
\end{definition}

Notice that the Lax-Oleinik admits a simpler definition
\begin{gather*}
T^t[u](q) = \inf_{p \in U} \{ u(p) + \mathcal{A}_{\phi,C}(p,q) \}.
\end{gather*}

We show in the first lemma that the penalized action between two points  is $C$-Lipschitz.

\begin{lemma} \label{Lemma:ActionAprioriEstimate}
Let $\phi : U \to \mathbb{R}$ be a bounded Lipschitz function and  $C\geq0$ be a  constant. Then for every $p,\tilde p, q,\tilde q \in U$
\begin{enumerate} 
\item \label{Item:ActionAprioriEstimate_1} $\big| \mathcal{A}^t_{\phi,C}(p,q) - C d_U(p,q) \big| \leq t \big( \Lip(\phi)\diam(U)+C\|V\|_\infty \big)$,
\item \label{Item:ActionAprioriEstimate_2} $\big| \mathcal{A}^t_{\phi,C}(p,q) - \mathcal{A}^t_{\phi,C}(p,\tilde q) \big| \leq C d_U(q,\tilde q)$,
\item \label{Item:ActionAprioriEstimate_3} $\big| \mathcal{A}^t_{\phi,C}(p,q) - \mathcal{A}^t_{\phi,C}(\tilde p,q) \big| \leq C d_U(p,\tilde p)$,
\end{enumerate}
where $\diam(U) = \sup_{p,q \in U}d_U(p,q)$.
\end{lemma}

\begin{proof}
{\it Item \ref{Item:ActionAprioriEstimate_1}.} Let $z: [0,t] \to U$ be a path joining $p$ and $q$. Then
\begin{multline*}
\Big| \int_0^t \big[\phi \circ z-\bar{\phi}_\Lambda+C \|V \circ z -z^{\prime}\| \big] \, ds - C \int_0^t \|z'(s)\| \, ds \Big| \\
\leq  t \big( \Lip(\phi)\diam(U) + C\|V\|_\infty \big).
\end{multline*}
We conclude the proof of item \ref{Item:ActionAprioriEstimate_1} by minimizing over such paths.

\medskip
{\it Item \ref{Item:ActionAprioriEstimate_2}.} Let $\tilde z : |0,t] \to U$ be a path joining $p$ and $\tilde q$, and $\tau>0$, $\epsilon>0$ supposed to be small. By definition of the distance $d_U$ as an infimum there exists a path $w_\tau : [t-\tau,t] \to U$ joining $\tilde z(t-\tau)$ and $q$ such that
\[
\int_{t-\tau}^t \|w'_\tau(s) \| \, ds \leq d_U(\tilde z(t-\tau),q) + \epsilon.
\]
Let $z:[0,t] \to U$ be the path obtained by concatenation of the restriction $\tilde z : [0,t-\tau] \to U$ and $w_\tau$. Then
\begin{align*}
\mathcal{A}_{\phi,C}^t(p,q) &\leq \int_0^t \big[\phi \circ z-\bar{\phi}_\Lambda+C \|V \circ z -z^{\prime}\| \big] \, ds \\
&\leq \int_0^{t} \big[\phi \circ \tilde z-\bar{\phi}_\Lambda+C \|V \circ \tilde z -\tilde z' \| \big] \, ds \\
&\quad + 2 \tau \big( \Lip(\phi)\diam(U) + C\|V\|_\infty \big) + \int_{t-\tau}^t C\big[ \|w'_\tau(s) \| + \| \tilde z'(s) \| \big] \, ds.
\end{align*}
Letting $\tau \to 0$ one obtains
\[
\mathcal{A}_{\phi,C}^t(p,q) \leq \int_0^{t} \big[\phi \circ \tilde z-\bar{\phi}_\Lambda+C \|V \circ \tilde z -\tilde z' \| \big] \, ds + C d_U(\tilde q,q) + \epsilon.
\]
Letting $\epsilon \to 0$ and minimizing over $\tilde z$ one obtains
\[
\mathcal{A}_{\phi,C}^t(p,q) \leq  \mathcal{A}_{\phi,C}^t(p,\tilde q) + C d_U(q,\tilde q).
\]
We conclude the proof of item \ref{Item:ActionAprioriEstimate_2} by permuting $q$ and $\tilde q$. 

\medskip
{\it Item \ref{Item:ActionAprioriEstimate_3}.}  The proof is similar. We set $\tau>0$ and $\epsilon>0$ small.  We choose a path $\tilde z : [0,T]$ joining $\tilde p$ and $q$ such that
\[
\int_0^t  \big[\phi \circ \tilde z-\bar{\phi}_\Lambda+C \|V \circ \tilde z -\tilde z' \| \big] \, ds \leq \mathcal{A}_{\phi,C}(\tilde p,q) + \epsilon.
\] 
We choose a path $w_\tau : [0,\tau] \to U$ that joins $p$ and $\tilde z(\tau)$ such that
\[
\int_0^\tau \| w_\tau'(s) \| \, ds \leq  d_U(\tilde z(\tau),p) +\epsilon.
\]
Define $z:[0,t] \to U$ as the concatenation of $w_\tau$ and the restriction $\tilde z : [\tau,t] \to U$. Then
\begin{align*}
\mathcal{A}_{\phi,C}^t(p,q) &\leq \int_0^t  \big[\phi \circ z-\bar{\phi}_\Lambda+C \|V \circ z -z' \| \big] \, ds \\
&\leq \mathcal{A}_{\phi,C}(\tilde p,q) + 2\tau \big( \Lip(\phi)\diam(U) +C \|V\|_\infty \big) \\
&\quad + \int_0^\tau C \big[ \| w_\tau'(s) \|  + \| \tilde z'(s) \| \big] \, ds + \epsilon.
\end{align*}
We conclude as before by letting first $\tau\to0$ and then $\epsilon \to 0$.
\end{proof}

We show in the second lemma basic properties of an action-like functional.

\begin{lemma} \label{Lemma:ContinuousLaxOleinikProperties}
Let $T$ be the Lax-Oleinik semigroup on $(U,\Lambda)$ of generator $\phi$ as defined in Definition~\ref{Definition:ContinuousLaxOleinikSemigroup}. Then for every $s,t >0$, every bounded function $ u,v : U \to \mathbb{R}$,
\begin{enumerate}
\item \label{Item:ContinuousLaxOleinikProperties_1} $T^t \circ T^s[u] = T^{t+s}[u]$,
\item \label{Item:ContinuousLaxOleinikProperties_2} $ u \leq v \ \Rightarrow \ T^t[u] \leq  T^t[v]$,
\item \label{Item:ContinuousLaxOleinikProperties_3} $\forall c \in \mathbb{R}, \ T^t[u+c] = T^t[u]+c$,
\item \label{Item:ContinuousLaxOleinikProperties_4} for every uniformly bounded family $(u_t)_{t>0}$, 
\[
\inf_{t>0} T^s[u_t](q) = T^s[\inf_{t>0}u_t](q).
\]
\item \label{Item:ContinuousLaxOleinikProperties_5} $\sup_{q \in U} | T^t[u](q) - T^t[v](q) | \leq  \sup_{q \in U} |u(q) - v(q) |$.
\item \label{Item:ContinuousLaxOleinikProperties_6} $T^t[u]$ is $C$-Lipschitz.
\end{enumerate}
\end{lemma}

\begin{proof}
{\it Part 1.} We prove item \ref{Item:ContinuousLaxOleinikProperties_1}. We rewrite the Lax-Oleinik operator
\[
T^t[u](q) = \inf_{p \in U} \big\{ u(p) + \mathcal{A}_{\phi,C}^t(p,q) \big\}.
\] 
The penalized action between two  points satisfies the inf-convolution property
\[
\mathcal{A}_{\phi,C}^{s+t}(p,q) = \inf_{r \in U} \big\{ \mathcal{A}_{\phi,C}^{s}(p,r) + \mathcal{A}_{\phi,C}^{t}(r,q) \big\}.
\]
Then by permuting the two infimum
\begin{align*}
T^{s+t}[u](q) &= \inf_{p \in U} \inf_{r \in U} \big\{ u(p) +  \mathcal{A}_{\phi,C}^{s}(p,r) + \mathcal{A}_{\phi,C}^{t}(r,q) \big\} \\
&= \inf_{r \in U} \big\{ T^s[u](r) + \mathcal{A}_{\phi,C}^{t}(r,q) \big\} = T^t \circ T^s[u](q).
\end{align*}
{\it Part 2.} Items \ref{Item:ContinuousLaxOleinikProperties_2}-- \ref{Item:ContinuousLaxOleinikProperties_6} are easily proved.
\end{proof}

\begin{proof}[\bf Proof of Theorem \ref{Theorem:ContinuousWeakKAMsolution}]
{\it Step 1.} Let $u : U \to \mathbb{R}$ be a bounded function and $t>0$. We show that $T^t[u]$ is $C$-Lipschitz uniformly in time $t$. Indeed for every $p,q,\tilde q$
\[
u(p)+\mathcal{A}_{\phi,C}^t(p,q) \leq u(p)+\mathcal{A}_{\phi,C}^t(p,\tilde q) + C d_U(q,\tilde q).
\]
Minimizing over $p$  and permuting $q$ and $\tilde q$ one obtains
\[
\big| T^t[u](q)-T^t[u](\tilde q) \big| \leq C d_U(q,\tilde q).
\]
{\it Step 2.} Let $v = \inf_{t>0}T^t[0]$. The function $v$ is well defined because of the positive Liv\v{s}ic criterion. We show that $T^s[v] \geq v$ for every $s>0$. Indeed using item \ref{Item:ContinuousLaxOleinikProperties_4} of Lemma \ref{Lemma:ContinuousLaxOleinikProperties}
\[
T^s[v] = T^s[\inf_{t>0}T^t[u]] = \inf_{t>0} T^s[T^t[u]] = \inf_{t>0} T^{s+t}[u] = \inf_{t> s} T^t[u] \geq v.
\]
In particular $s \in (0,+\infty) \mapsto T^s[u](q) \in\mathbb{R}$ is non decreasing.

\medskip
{\it Step 3.} Let $u = \sup_{s>0}T^s[v] = \lim_{n\to+\infty}T^n[v]$. We show that $u$ is a weak KAM solution, that is,
\[
\forall\, q \in U, \ \forall\, t>0, \ T^t[u](q)=u(q).
\]
Let $t>0$ fixed. On the one hand, $u \geq T^n[v]$ for every $n\geq1$, then $T^t[u] \geq T^{t+n}[v]$ for every $n\geq1$. Letting $n\to+\infty$ and using the continuity of $T^t$ for the uniform topology, one obtains $T^t[u] \geq u$. On the other hand, let $q \in U$ fixed, and $(\epsilon_n)_{n\geq1}$ be a sequence of positive real numbers tending to 0, then there exists $p_n \in U$ such that
\[
T^n[v](p_n)+ \mathcal{A}_{\phi,C}^t(p_n,q) \leq T^{t+n}[v](q) + \epsilon_n  \leq u(q) + \epsilon_n.
\]
Let $p \in \bar U$ be an accumulation point of $(p_n)_{n\geq1}$ and $\tilde p \in U$ supposed to be closed to $p$. As $T^n[v]$ and $\mathcal{A}_{\phi,C}^t(\cdot,q)$ are $C$-Lipschitz, one obtains
\[
T^n[v](\tilde p) + \mathcal{A}_{\phi,C}^t(\tilde p,q) \leq 2C d_U(\tilde p,p_n) + u(q)+\epsilon_n.
\]
Letting $n\to+\infty$ one obtains, $u(\tilde p) = \lim_{n \to+\infty} T^n[v](\tilde p)$ and
\[
T^t[u](q) \leq u(\tilde p) + \mathcal{A}_{\phi,C}^t(\tilde p,q) \leq 2C d_U(\tilde p,p) + u(q).
\]
Minimizing over $\tilde p$ we have $T^t[u](q) \leq u(q)$ and therefore $T^t[u]=u$.
\end{proof}

We conclude the proof of our main result Theorem \ref{Theorem:ContinuousSubactionExistence}. We use a notion of regularizing operator that has been suggested to us by Bony \cite{Bony2022}. Theorem \ref{Theorem:ContinuousWeakKAMsolution} shows the existence of a Lipschitz continuous integrated subaction $u_0:  \Omega \to \mathbb{R}$ (item \ref{Label:ContinuousErgodicOptimization_2} of Definition \ref{Definition:ContinuousErgodicOptimization}) defined on an open set $\Omega \supseteq \Lambda$. We cover $\Lambda$ by a finite set of open sets $(D_i)_{i=1}^N$. Let $U := \bigcup_{i=1}^N D_i$. We then construct a family of nonlinear operators $(R_i)_{i=1}^N$ that are regularizing on each $D_i$ in the following sense:
\begin{itemize}
\item if $u$ is a Lipschitz continuous integrated subaction on $\Omega$, then $v:=R_i[u]$ is also a Lipschitz continuous integrated subaction on $\Omega$,
\item if $u$ is locally a subaction at $x \in \Omega$, then $v$ is locally a subaction at $x$,
\item $v$ is  locally a subaction at every point of $D_i$.
\end{itemize}
We say that $u$ is {\it locally a subaction} at $x \in U$ if there exists a neighborhood $W$ of $x$ such that 
\[
\begin{cases}
\text{$u$ is differentiable along the flow in $W$}, \\
\text{$\mathcal{L}_V[u]$ is Lipschitz continuous in $W$}, \\
\forall\, y \in W, \ \mathcal{L}_V[u] (y) \leq \phi(y) - \bar \phi_\Lambda.
\end{cases}
\]
Then $v=R_N \circ \cdots \circ R_1[u_0]$ is a subaction on $U=\bigcup_{i=1}^kD_i \subset \Omega$.

\begin{proof}[\bf Proof of Theorem \ref{Theorem:ContinuousSubactionExistence}]
We consider a family of local flow boxes $(D''_i)_{i=1}^N$ indexed by a finite subset of points $\{x_1, \ldots,x_N\} \subset \Lambda$
\[
D''_i = \gamma_{x_i}(\tilde D''), \quad \tilde D'' = (-2\epsilon,\tau+2\epsilon) \times B_{x_i}(3\epsilon).
\] 
(See item \ref{Item:AdaptedLocalFlowBoxes_2} of  Definition \ref{Definition:AdaptedLocalFlowBoxes} for the definition of a family of local flow boxes.)  We also note
\begin{gather*}
D'_i= \gamma_{x_i}(\tilde D_i'), \quad \tilde D_i' = (-\epsilon,\tau+\epsilon) \times B_{x_i}(2\epsilon), \\
D_i = \gamma_{x_i}(\tilde D), \quad \tilde D_i = (0,\tau) \times B_{x_i}(\epsilon), \\
D_i \subset D'_i \subset D''_i,
\end{gather*}
choose $\epsilon>0$ small enough and the number of points $N$ large enough so that 
\[
\Lambda \subset \bigcup_{i=1}^N  D_i\ \ \text{and} \ \ \bigcup_{i=1}^N D''_i \subset \Omega.
\] 
We construct a regularizing  operator $u \mapsto R_{i}[u]$ for every $i\in\llbracket 1,N\rrbracket$, possessing the following properties: if $u$ is a Lipschitz continuous integrated subaction on $\Omega$ and $v = R_{i}[u]$ then
\begin{itemize}
\item $v$ is a Lipschitz continuous integrated subaction on $\Omega$,
\item $v$ is locally a subaction at every $x \in D_i$,
\item if $u$ is locally a subaction at $x \in\Omega$ then $v$ is locally a subaction at $x$,
\item $u$ and $v$ coincide on a neighborhood of $\Omega \setminus D''_i$.
\end{itemize}
Let $\alpha_i : B_{x_i}(3\epsilon) \to [0,1]$ be a smooth function satisfying $\Supp(\alpha_i) \subset B_{x_i}(2\epsilon)$ and $\alpha_i=1$ on $B_{x_i}(\epsilon)$. Let $\beta_i : (-2\epsilon, \tau+2\epsilon) \to [0,1]$ be a smooth function satisfying $\Supp(\beta_i) \subset (-\epsilon, \tau+\epsilon)$ and $\beta_i=1$ on $(0,\tau)$. We define $v=R_i[u]$ in the following way
\[
\begin{cases}
v(z) = u(z), & \forall\, z \in \Omega \setminus \overline{D_i'}, \\
v(z) = \tilde v_i \circ \gamma_{x_i}^{-1}, & \forall\, z \in D''_i,
\end{cases}
\]
for some $\tilde v_i : \tilde D_i'' \to \mathbb{R}$  that  coincides with $\tilde u_i := u \circ \gamma_{x_i}^{-1}$ on $\tilde D_i'' \setminus \overline{\tilde D_i'}$. 

We now construct $\tilde v_i$. For every $q\in  B_{x_i}(3\epsilon)$,  $t \mapsto \tilde u_i (t,q)$ is Lipschitz continuous, $\partial_t \tilde u_i (t,q)$ exists for almost every $t$ and in the distribution sense. Moreover the two notions of differentiability, Lebesgue almost everywhere and in the distribution sense, coincide.  We may moreover assume that $\partial_t \tilde u_i$ is Borel with respect to $(t,q) \in \tilde D_i''$.  Let $\tilde \phi_i := \phi \circ \gamma_{x_i} - \bar\phi_\Lambda$. Define for every $(t,q) \in \tilde D''_i$
\begin{gather*}
\tilde v_i (t,q) = (1-\alpha_i(q)) \tilde u_i (t,q) + \alpha_i(q) \tilde w_i(t,q), 
\end{gather*}
where 
\begin{multline}
\tilde w_i(t,q) = \tilde u_i(t,q) + \int_{-\epsilon}^t \beta_i(s) \big( \tilde\phi_i(s,q)   -  \partial_t \tilde u_i(s,q) \big)   \, ds \\
- \frac{\int_{-\epsilon}^t \beta_i(s) \, ds}{\int_{-\epsilon}^{\tau+\epsilon}\beta_i(s) \, ds}  \int_{-\epsilon}^{\tau+\epsilon} \beta_i(s) \big( \tilde \phi_i(s,q)  - \partial_t \tilde u_i(s,q) \big)  \, ds. \label{Equation:ProofMainTheorem_1}
\end{multline}
The first observation is the fact that $\tilde w_i$ is differentiable along the flow on $\tilde D_i$ and that $\partial_t \tilde w_i$ is Lipschitz continuous. Indeed integrating by part, as $\beta_i(t) = 1$ for $t \in (0,\tau)$, one obtains
\begin{align*}
\tilde u_i(t,q) &- \int_{-\epsilon}^t \beta_i(s) \partial_t \tilde u_i (s,q)   \, ds  \\
&= ( 1 -  \beta_i (t)) \tilde u_i (t,q) + \int_{-\epsilon}^t  \beta_i'(s) \tilde u_i (s,q)  \, ds  = \int_{-\epsilon}^t  \beta_i'(s) \tilde u_i (s,q)  \, ds, \\
\partial_t \tilde w_i(t,q) &= \tilde \phi_i (t,q) - \frac{\int_{-\epsilon}^{\tau+\epsilon} \beta_i (s) \big( \tilde \phi_i (s,q)  - \partial_t u_i(s,q) \big) \, ds}{\int_{-\epsilon}^{\tau+\epsilon}\beta_i(s) \, ds} \\
&= \tilde \phi_i (t,q) - \frac{\int_{-\epsilon}^{\tau+\epsilon} \big[ \beta_i (s) \tilde \phi_i (s,q)  +\beta_i'(s) u_i(s,q) \big] \, ds}{\int_{-\epsilon}^{\tau+\epsilon}\beta_i(s) \, ds}.
\end{align*}
The second observation is the fact that $\tilde w_i$ is an integrated subaction on $\tilde D''_i$ (and in particular a subaction on $\tilde D_i$). Indeed as $u$ is an integrated subaction, one observes for every $(t,q) \in \tilde D''_i$
\begin{gather}
\tilde \phi_i(t,q)  - \partial_t \tilde  u_i(t,q) \geq  0. \label{Equation:ProofMainTheorem_2}
\end{gather}
Then for every $t_1 < t_2$ in $(-2\epsilon, \tau+2\epsilon)$, for every  $q \in B_{x_i}(3\epsilon)$, as $\beta_i$ takes values in $[0,1]$, using \eqref{Equation:ProofMainTheorem_1} and \eqref{Equation:ProofMainTheorem_2}, we have
\begin{align*}
\int_{t_1}^{t_2} \beta_i(s)& \big( \tilde\phi_i(s,q)  -  \partial_t \tilde u_i(s,q) \big)  \, ds  \\
&\leq \int_{t_1}^{t_2} \big( \tilde\phi_i(s,q)  -  \partial_t \tilde u_i(s,q) \big)  \, ds  \\
&= \int_{t_1}^{t_2} \tilde\phi(s,q)  \, ds    - \big( \tilde u_i(t_2,q) - \tilde u_i(t_1,q) \big), \\
\tilde w_i(t_2,q) &- \tilde w_i(t_1,q) \\
&\leq \int_{t_1}^{t_2} \tilde\phi_ i(s,q) \, ds 
 - \frac{\int_{t_1}^{t_2} \beta_i(s) \, ds}{\int_{-\epsilon}^{\tau+\epsilon}\beta_i(s) \, ds}  \int_{-\epsilon}^{\tau+\epsilon} \beta_i(s) \big( \tilde \phi_i (s,q)  - \partial_t \tilde u_i(s,q) \big)  \, ds \\
&\leq \int_{t_1}^{t_2} \tilde \phi_i (s,q) \, ds.
\end{align*}
By construction $\tilde v_i$ coincides with $\tilde u_i$ on $\tilde D''_i \setminus \overline{\tilde D'_i}$. Moreover if $\tilde u_i$ is locally a subaction at $(t,q) \in \tilde D''_i$, then $\tilde v_i$ is locally a subaction at $(t,q)$.

Lastly if $(t,q) \in \tilde D_i''$ and $\tilde u $ is locally a subaction at $(t,q)$, then by construction of $\tilde w$, $\tilde v$ is also locally a subaction at $(t,q)$.

We now conclude the proof of Theorem \ref{Theorem:ContinuousSubactionExistence}.  Let $u_0$ be a Lipschitz continuous integrated subaction given by Theorem \ref{Theorem:ContinuousWeakKAMsolution}. By induction 
\[
u_k := R_{{k}} \circ \cdots \circ  R_{{1}}[u_0]
\]
is a Lipschitz continuous integrated subaction that is locally a subaction at every point of $\bigcup_{i=1}^k D_{i}$. The function $u_N : \bigcup_{i=1}^N D_{i} \to \mathbb{R}$ is the desired Lipschitz continuous subaction.
\end{proof}

\begin{proof}[\bf Proof of Corollary \ref{Corollary:SimultaneousBounds}] Theorem \ref{Theorem:ContinuousSubactionExistence} implies there exists  a Lipschitz continuous subaction $u_1$ such that
\begin{gather*}
\phi_1 := \phi - \mathcal{L}_V[u_1] \geq  \bar\phi_\Lambda.
\end{gather*}
As $\phi$ and $\phi_1$ have both the same ergodic maximizing value, Theorem \ref{Theorem:ContinuousSubactionExistence} again shows that there exists a Lipschitz continuous subaction $u_2$ such that
\begin{gather*}
\phi_1 - \mathcal{L}_V[u_2] \leq  \bar{\bar\phi}_\Lambda.
\end{gather*}
We may assume that $u_2\leq0$. Let $C=\Osc(u_2)$ be the oscillation of $u_2$ and $T= 4C/(\bar{\bar\phi}_\Lambda - \bar\phi_\Lambda)$. If $u_2$ is constant the proof is finished. Assume $C>0$. Define the following two nonempty compact sets
\begin{gather*}
{\bar K} := \Big\{ x \in \Omega : \int_0^T \phi_1 \circ f^s(x) \, ds  \leq C + T \bar\phi_\Lambda \Big\}, \\
{\bar{\bar K}} := \Big\{ x \in \Omega : \int_0^T \big(\phi_1 - \mathcal{L}_V[u_2] \big) \circ f^s(x) \, ds \geq T\bar{\bar\phi}_\Lambda -C \Big\}.
\end{gather*}
Indeed, we prove that $\bar K \neq \emptyset$ and the fact $\bar{\bar K} \neq \emptyset$ can be obtained similarly. By  contradiction, we have for every $x\in \Omega$,
\[
\int_0^T \phi_1\circ f^s(x)ds \geq T \left(\bar{\phi}_\Lambda + \frac{C}{T} \right). 
\]
For any integer $N\geq 1$, by concatenating $x, f^T(x), f^{2T}(x), \cdots, f^{NT}(x)$, one obtains
\[
\forall~x\in \Lambda, \forall~N\geq 1, \quad \int_0^{NT}  \phi_1\circ f^s(x)ds \geq NT \left(\bar{\phi}_\Lambda + \frac{C}{T} \right). 
\]
We then get the contradiction
\[
\bar{\phi}_\Lambda = \lim_{N\rightarrow +\infty} \inf_{x\in \Lambda} \frac{1}{NT} \int_0^{NT}  \phi_1\circ f^s(x)ds \geq \bar{\phi}_\Lambda + \frac{C}{T}.
\]

Notice first that ${\bar K} \cap {\bar{\bar K}} = \emptyset$. Indeed the existence of a point $x \in {\bar K} \cap {\bar{\bar K}}$ would imply the following absurd inequality
\begin{multline*}
C \geq u_2(x) - u_2 \circ f^T(x) = \int_0^T -\mathcal{L}_V[u_2] \circ f^s(x) \, ds \\
\geq  T( \bar{\bar\phi}_\Lambda - \bar \phi_\Lambda) - 2C = 2C.
\end{multline*}
We choose a smooth function $\theta : \Omega \to [0,1]$ such that
\begin{gather*}
\forall\, x \in {\bar K}, \ \theta(x) =0, \quad  \forall\, x \in {\bar{\bar K}}, \ \theta(x) = 1.
\end{gather*}
Let 
\begin{gather*}
\phi_2  := \phi_1 - \mathcal{L}[\theta u_2].
\end{gather*}
We show that 
\begin{gather*}
\forall\, x \in \Omega, \ \int_0^T \phi_2 \circ f^s(x) \, ds \geq T\bar\phi_\Lambda.
\end{gather*} 
Either $x \in {\bar K}$ then $(\theta u_2)(x)=0$, $(\theta u_2) \circ f^T(x) \leq 0$, 
\begin{align*}
\int_0^T \phi_2 \circ f^s(x) \, d &=   \int_0^T \phi_1 \circ f^s(x) \, ds + (\theta u_2)(x) - (\theta u_2)\circ f^T(x) \\
&\geq \int_0^T \phi_1 \circ f^s(x) \, ds  \geq T \bar\phi_\Lambda.
\end{align*}
Or $x\not\in {\bar K}$, then
\begin{gather*}
\int_0^T \phi_2 \circ f^s(x) \, ds  \geq \int_0^T \phi_1 \circ f^s(x)  \, ds - C   > T \bar\phi_\Lambda.
\end{gather*}
We show that 
\begin{gather*}
\forall\, x \in \Omega, \ \int_0^T \phi_2 \circ f^s(x) \, ds \leq T \bar{\bar\phi}_\Lambda.
\end{gather*}
Either $x \in {\bar{\bar K}}$, then $((1-\theta)u_2) \circ f^T(x) \leq 0$, $((1-\theta)u_2)(x) = 0$,
\begin{align*}
\int_0^T &\phi_2 \circ f^s(x) \, ds \\
&= \int_0^T \big( \phi_1 -\mathcal{L}_V[u_2] \big) \circ f^s(x) \, ds + \int_0^T ((1-\theta)u_2) \circ f^s(x) \, ds \\
&= \int_0^T \big( \phi_1 -\mathcal{L}_V[u_2] \big) \circ f^s(x) \, ds + ((1-\theta)u_2) \circ f^T(x) - ((1-\theta)u_2)(x) \\
&\leq  T \bar{\bar \phi}_\Lambda.
\end{align*}
Or $x \not\in {\bar{\bar K}}$, then
\begin{align*}
\int_0^T \phi_2 \circ f^s(x) \, ds &\leq \int_0^T \big( \phi_1 -\mathcal{L}_V[u_2] \big) \circ f^s(x) \, ds +  \Osc((1-\theta)u_2)) \\
&\leq T \bar{\bar \phi}_\Lambda -C +  \Osc((1-\theta)u_2)) \leq T \bar{\bar \phi}_\Lambda.
\end{align*}
Let 
\begin{gather*}
u_3(x) = -\frac{1}{T} \int_0^T \Big[ \int_0^t \phi_2 \circ f^s(x) \, ds \Big] dt.
\end{gather*}
Then for every $x\in\Omega$
\begin{gather*}
\mathcal{L}_V[u_3](x) = \frac{1}{T} \int_0^T \big(\phi_2(x) -  \phi_2 \circ f^t(x) \big) \, dt =  \phi_2(x) - \frac{1}{T} \int_0^T \phi_2 \circ f^s(x) \, ds , \\
\phi_2(x) - \mathcal{L}_V[u_3] =\frac{1}{T} \int_0^T \phi_2 \circ f^s(x) \, ds \in [\bar\phi_\Lambda, \bar{\bar\phi}_\Lambda], \\
\bar\phi_\Lambda \leq \phi - \mathcal{L}_V[u_1+u_2+u_3] \leq \bar{\bar\phi}_\Lambda.\qedhere
\end{gather*}
\end{proof}

\vfill
\pagebreak
\appendix
\appendixpageoff
\addappheadtotoc

%%%%%%%%%%%%%%%%
\section{Local hyperbolic flows}
%%%%%%%%%%%%%%%%
\label{Appendix:LocalHyperbolicFlows}

We review in this section the local theory of hyperbolic flows. Good monographs on the subject can be found in  Hasselblatt, Katok  \cite{HasselblattKatok1995}, Bonatti, Diaz, Viana \cite{BonattiDiazViana2005}, and Fisher, Hasselblatt \cite{FisherHasselblatt2019}.

The notion of flow boxes is standard. We need nevertheless to describe precisely  the Poincar\'e sections and the constants of hyperbolicity of  the corresponding Poincar\'e maps. As the Poincar\'e maps are also uniformly hyperbolic, we recall the notion of adapted local hyperbolic maps as written in Appendix A, Definition A.1 in \cite{SuThieullenYu2021}.

%%%%%%%%%%%%%%%%%
\subsection{Adapted local flow boxes}
%%%%%%%%%%%%%%%%%

We consider in the following definition a local Lipschitz map $f: B(\rho) \to \mathbb{R}^d$ where $B(\rho)$ is a ball of $\mathbb{R}^d$ for some norm $| \cdot |$ and the target space is  $\mathbb{R}^d$ equipped with a possibly different norm $\| \cdot \|$. We assume that the source and target spaces are direct sums of an unstable space and a stable space
\[
\mathbb{R}^d = E^u \oplus E^s, \quad \mathbb{R}^d= \tilde E^u \oplus \tilde E^s.
\]
We denote by $P^u : \mathbb{R}^d \to E^u$, $P^s : \mathbb{R}^d \to E^s$, and $\tilde P^u : \mathbb{R}^d \to \tilde E^u$ and $\tilde P^s : \mathbb{R} \to \tilde E^s$, the corresponding projections. We assume that the two norms are adapted to the splittings in the sense
\[
\left\{ \begin{array}{l}
\forall v,w \in E^u \times E^s, \ |v+w| = \max(|v|,|w|), \\
\forall v,w \in \tilde E^u \times \tilde E^s, \ \|v+w\| = \max(\|v\|, \|w\|).
\end{array}\right.
\]
The balls are denoted by $B(\rho)$, $B^u(\rho)$, $B^s(\rho)$ in the source space, and by $\tilde B(\rho)$, $\tilde B^u(\rho)$,  $\tilde B^s(\rho)$ in the target space. In particular $B(\rho) = B^u(\rho) \times B^s(\rho)$, $\tilde{B}(\rho) = \tilde{B}^u(\rho) \times \tilde{B}^s(\rho)$.

\begin{definition}[Adapted local hyperbolic map] \label{Definition:AdaptedLocalHyperbolicMap}
Let $(\sigma^s,\sigma^u,\eta,\rho)$ be positive real numbers called {\it constants of hyperbolicity} satisfying
\[
\sigma^u > 1 > \sigma^s, \quad \eta < \min\Big( \frac{\sigma^u-1}{6},\frac{1-\sigma^s}{6} \Big), \quad \epsilon(\rho) := \rho \min\Big( \frac{\sigma^u-1}{2},\frac{1-\sigma^s}{8}\Big).
\]
An {\it adapted local hyperbolic map with respect to the two norms and the constants of hyperbolicity } is a set of data $(f,A, E^{u/s}, \tilde E^{u/s}, |\cdot|,\| \cdot\|)$ such that:
\begin{enumerate}
\item $f : B(\rho) \to \mathbb{R}^d$ is a Lipschitz map,
\item  $A:\mathbb{R}^d \to \mathbb{R}^d$ is a linear map which may not be invertible and is defined into block matrices
\[
A = \begin{bmatrix}
A^u & D^u \\ D^s &  A^s
\end{bmatrix}, \quad 
\left\{\begin{array}{l}
(v,w) \in E^u \times E^s, \\
A(v+w) = \tilde v + \tilde w,
\end{array}\right.
\ \Rightarrow \
\left\{\begin{array}{l}
\tilde v = A^u v + D^u w \in \tilde E^u, \\ \tilde w = D^s v  + A^s w \in \tilde E^s,
\end{array}\right.
\]
that satisfies
\[
\left\{\begin{array}{l}
\forall \, v \in E^u, \ \|A^u v\| \geq \sigma^u \|v\|, \\
\forall \, w \in E^s, \ \|A^s w\| \leq \sigma^s \|w\|,
\end{array}\right. \quad\text{and}\quad
\left\{\begin{array}{ll}
\|D^u\| \leq \eta, & \Lip(f-A) \leq \eta, \\ \|D^s\| \leq \eta, & \|f(0)\| \leq \epsilon(\rho),
\end{array}\right.
\]
where  the $\Lip$ constant is computed using the two norms $| \cdot |$ and $\| \cdot \|$.
\end{enumerate}
\end{definition}

We consider now a $C^1$ flow $(M,V,f)$  and an $f$-invariant compact hyperbolic set as defined in Definition \ref{Definition:LocallyHyperbolicFlow}.

\begin{definition}[Adapted local flow boxes]
\label{Definition:AdaptedLocalFlowBoxes}
A family of adapted local flow boxes is the set of data $\Gamma=(\Gamma,\Sigma, E, N,F,A)$ and the set of constants $(\sigma^u,\sigma^s,\eta,\rho,\tau)$ satisfying the following properties.
\begin{enumerate}
\item \label{Item:AdaptedLocalFlowBoxes_1} The constants $(\sigma^u,\sigma^s,\eta,\rho, \tau)$ are chosen so that $\tau>0$ and
\begin{gather*}
\exp \Big(\frac{\tau \lambda^s}{2} \Big) < \sigma^s < 1 < \sigma^u < \exp \Big(\frac{\tau\lambda^u}{2} \Big) \\ 
\eta < \min \Big( \frac{\sigma^u-1}{6}, \frac{1-\sigma^s}{6} \Big), \quad  \epsilon(\rho) := \rho \min\Big( \frac{\sigma^u-1}{2},\frac{1-\sigma^s}{8}\Big).
\end{gather*}

\item \label{Item:AdaptedLocalFlowBoxes_2}  $\Gamma := (\gamma_x)_{x\in\Lambda}$ is a parametrized family of $C^1$  diffeomorphisms such that
\[
\gamma_x : \tilde D := (-\tau,2\tau) \times B(1) \twoheadrightarrow D_x \subseteq M, \quad  B(1) \subset \mathbb{R}^d,
\] 
are onto   $D_x := \gamma_x( \tilde D_x)$ that is assumed to be open, where $B(1)$ is the euclidean  unit ball. Moreover for every $x \in\Lambda$
\begin{enumerate}
\item  $\gamma_x(0,0)=x$, 
\item the $C^1$ norm of $\gamma_x$ and $\gamma_x^{-1}$ is uniformly bounded with respect to $x\in\Lambda$, that is denoted by $\Lip(\Gamma)$ and chosen so that  $\Lip(\Gamma)\geq1$; the {\it local coordinates} of a point $q \in D_x$ are denoted by
\[
\gamma_{x}^{-1}(q) = (\tau_x(q),\pi_x(q)), \quad 
\begin{cases}
\tau_x : D_x \to (-\tau,2\tau), \\
\pi_x : D_x \to B(1),
\end{cases}
\]
\item the flow $f$ is locally conjugated to the constant horizontal flow generated by the vector field $e_1=(1,0)$,
\[
\left\{ \begin{array}{l}
\forall\, (s,u) \in (-\tau,2\tau) \times B(1), \ T_{s,u}\gamma_x  (1,0) = V \circ \gamma_x(s,u), \ \text{(equivalently)} \\
\forall\, (s,u) \in (-\tau,2\tau) \times B(1), \ f^s \circ \gamma_x(0,u) = \gamma_x(s,u).
\end{array}\right.
\]
\end{enumerate}
The pair $(D_x, \gamma_x^{-1})$ is called a {\it local flow box}.

\item \label{Item:AdaptedLocalFlowBoxes_3} $\Sigma := (\Sigma_x)_{x \in\Lambda}$, where $\Sigma_x:=\gamma_x(\{0\} \times B(1))$, is the family of  {\it local Poincar\'e sections at $x$}. Let  $\gamma_{x,0} : B(1) \to \Sigma_x$ and $\gamma_{x,0}^{-1} : \Sigma_x \to  B(1)$ be the restriction of $\gamma_x$ to $\{0\} \times B(1)$ and its  inverse to $\Sigma_x$.

\item \label{Item:AdaptedLocalFlowBoxes_3_bis} $E = (E_x^{u/s})_{x \in\Lambda}$ is a parametrized family of splitting $\mathbb{R}^d = E_x^u \oplus E_x^s$  obtained by  pulling backward  the corresponding splitting $T_xM =E_\Lambda^u(x) \oplus E_{\Lambda}^0(x) \oplus E_\Lambda^s(x)$ by the tangent map $T_0\gamma_{x}$ at the origin of $\mathbb{R}^d$,
\[
T_{0}\gamma_x(\{0\} \times E_x^{u/s}) = E_\Lambda^{u/s}(x).
\] 
We denote by $P_x^u$, $P_x^s$  the corresponding projections.

\item \label{Item:AdaptedLocalFlowBoxes_4} $N:=(\| \cdot \|_x)_{x \in\Lambda}$ is a family of $C^0$ norms on $\mathbb{R}^d$ adapted to the splitting $\mathbb{R}^d = E_x^u \oplus E_x^s$. $B_x(\rho)$ denotes the  ball centered at the origin and of radius $\rho$ with respect to the norm $\| \cdot \|_x$. We assume $\rho$ small enough so that 
\[
B_x(\rho) \subset B(1).
\]
The norms $\| \cdot \|_x$ are called {\it local norms}.

\item \label{Item:AdaptedLocalFlowBoxes_5} We assume that the two constants $\rho < \frac{1}{3}\tau$ are chosen small enough so that, for every $x,y \in \Lambda$ satisfying 
\[
y \in \gamma_x((\tau-\rho,\tau+\rho) \times B_x(\rho)),
\] 
there exists a $C^1$ function ${\tilde\tau_{x,y}} : B_x(\rho) \to (0,2\tau)$ such that
\begin{gather*}
\left\{\begin{array}{l}
\forall\, q \in B_x(\rho), \ \gamma_x((-\tau,2\tau) \times \{q\}) \cap \Sigma_y \ \ \text{is a singleton $z_{x,y}(q)$},\\
\forall\, q \in B_x(\rho), \  z_{x,y}(q)=\gamma_x({\tilde\tau_{x,y}}(q),q) = f^{{\tilde\tau_{x,y}}(q)}( \gamma_x(0,q))  \in \Sigma_y, \\
\| \nabla {\tilde\tau_{x,y}}\|_x \leq 1.
\end{array}\right.
\end{gather*}
The function ${\tilde\tau_{x,y}}$ is  called {\it local return time}. We use also the notation
\[
\tau_{x,y} = \tilde\tau_{x,y} \circ \gamma_{x,0}^{-1} : \gamma_{x,0}(B_x(\rho)) \subset \Sigma_x \to \mathbb{R}.
\]

\item \label{Item:AdaptedLocalFlowBoxes_6}  $F:=(f_{x,y})_{x,y \in\Lambda}$ is the  family of maps $f_{x,y} : B_x(\rho) \to \gamma_{y,0}^{-1}(\Sigma_y) = B(1)$,  parametrized by the couples $(x,y)$ satisfying
\[
y \in \gamma_x((\tau-\rho,\tau+\rho) \times B_x(\rho)),
\] 
 and defined by
\[
\forall\, v\in B_x(\rho), \ f_{x,y}(v) = \gamma_{y,0}^{-1} \circ f^{{\tilde\tau_{x,y}}(v)} \circ \gamma_{x,0}(v) : B_x(\rho) \to B(1).
\]
In the condition on $(x,y)$, $y$ should be thought close to $\gamma_x(\{\tau\} \times B_x(\rho))$. 
The map $f_{x,y}$ is called  the {\it local Poincar\'e map} and corresponds to the Poincar\'e map $f^{{\tau_{x,y}}(\cdot)}(\cdot)$ between the two sections $\Sigma_x,\Sigma_y$.

\item \label{Item:AdaptedLocalFlowBoxes_7} $A := (A_{x,y})_{x,y \in\Lambda}$ is the  family of tangent maps of $f_{x,y}$ at the origin, where $x,y$ satisfy $y \in \gamma_x((\tau-\rho,\tau+\rho) \times B_x(\rho))$, 
\[
A_{x,y} := Df_{x,y}(0),
\]

\item \label{Item:AdaptedLocalFlowBoxes_7bis} A couple of points $(x,y) \in \Lambda \times \Lambda$ is said to be {\it $\Gamma$ forward admissible}, and we write  $x \overset{\Gamma}{\to} y$,  if
\[
y \in \gamma_x((\tau-\rho,\tau+\rho) \times B_x(\rho)) \ \text{and} \ \ f_{x,y}(0) \in B_y(\epsilon(\rho)).
\]

\item \label{Item:AdaptedLocalFlowBoxes_8} For every $\Gamma$ forward admissible $x \overset{\Gamma}{\to} y$ ,  the set of data 
\begin{gather*}
(f_{x,y},A_{x,y}, E_x^{u/s}, E_y^{u/s}, \| \cdot \|_x, \| \cdot \|_y)
\end{gather*} 
and constants $(\sigma^u,\sigma^s,\eta,\rho)$ is an adapted local hyperbolic map as in  Definition \ref{Definition:AdaptedLocalHyperbolicMap}.
\end{enumerate}
\end{definition}

In less precise terms, if $(x,y)$ is $\Gamma$ forward admissible, the Poincar\'e map on the manifold  $[z\in \gamma_x(B_x(\rho)) \subseteq \Sigma_x \mapsto f^{{\tilde\tau_{x,y}}(z)}(z) \in\Sigma_y]$ is well defined; the local Poincar\'e map $[v \in B_x(\rho) \mapsto f_{x,y}(v) \in B(1)]$ is  a local hyperbolic map as in Definition \ref{Definition:AdaptedLocalHyperbolicMap}.

%%%%%%%%%%%%%%
\section*{Acknowledgement}
%%%%%%%%%%%%%%

We would like to thank J-.F. Bony for helping us at the last part of the proof of Theorem \ref{Theorem:ContinuousSubactionExistence}.

\end{document}